\documentclass[12pt]{amsart}
\usepackage[english]{babel}
\usepackage{amssymb}
\usepackage{amsmath}
\usepackage{mathtools}
\usepackage[pdftex]{graphicx}
\usepackage{times}
\usepackage{amsfonts}
\usepackage{amsthm}
\usepackage{latexsym}
\newtheorem{definition}{Definition}
\newtheorem{theorem}[definition]{Theorem}

\newtheorem{lemma}[definition]{Lemma}

\DeclareMathOperator{\sign}{sign}
\DeclareMathOperator{\dist}{dist}
\DeclareMathOperator{\closure}{clos}

\DeclareMathOperator{\conv}{conv}
\DeclareMathOperator{\supp}{supp}

\DeclarePairedDelimiter{\ceil}{\lceil}{\rceil}
\DeclarePairedDelimiter{\floor}{\lfloor}{\rfloor}

\begin{document}

\def\C{\mathbb {C}}
\def\N{\mathbb {N}}
\def\R{\mathbb {R}}
\def\S{\mathbb {S}}
\def\H{\mathbb {H}}
\def\Z{\mathbb {Z}}
\newcommand{\map}[3]{\ensuremath{#1\colon #2\to #3}}
\def\h{\mbox{\boldmath$h$}}

\def\qed{\rule{1ex}{1ex}}

\title[Mesh-independent a priori bounds]{Mesh-independent a priori bounds for nonlinear elliptic finite difference boundary value
  problems}

\author{P.J.~McKenna}
\address{P.J.~McKenna \hfill\break
Department of Mathematics, University of Connecticut, \hfill\break
Storrs, CT-06269, USA}
\email{mckenna@math.uconn.edu}

\author{W.~Reichel}
\address{W. Reichel \hfill\break
Institut f\"ur Analysis, Karlsruhe Institute of Technology (KIT), \hfill\break
D-76128 Karlsruhe, Germany}
\email{wolfgang.reichel@kit.edu}

\author{A.~Verbitsky}
\address{A. Verbitsky \hfill\break
DFG Research Training Group 1294\hfill\break
Analysis, Simulation and Design of Nanotechnological Processes, \hfill\break
Karlsruhe Institute of Technology (KIT), \hfill\break
D-76128 Karlsruhe, Germany}
\email{anton.verbitsky@kit.edu}

\date{\today}

\subjclass[2000]{Primary: 35J66, 39A14; Secondary: 34B18}
\keywords{finite difference equations, nonlinear boundary value problems, critical exponent, a-priori bounds}

\begin{abstract} In this paper we prove mesh independent a priori $L^\infty$-bounds for positive solutions of the finite difference boundary value problem 
$$
-\Delta_h u = f(x,u) \mbox{ in } \Omega_h, \quad u=0 \mbox{ on } \partial\Omega_h,
$$
where $\Delta_h$ is the finite difference Laplacian and $\Omega_h$ is a discretized $n$-dimensional box. On one hand this completes a result of \cite{mcr} on the asympotic symmetry of solutions of finite difference boundary value problems. On the other hand it is a finite difference version of a critical exponent problem studied in \cite{mcr2}. Two main results are given: one for dimension $n=1$ and one for the higher dimensional case $n\geq 2$. The methods of proof differ substantially in these two cases. In the 1-dimensional case our method resembles ode-techniques. In the higher dimensional case the growth rate of the nonlinearity has to be bounded by an exponent $p<\frac{n}{n-1}$ where we believe that $\frac{n}{n-1}$ plays the role of a critical exponent. Our method in this case is based on the use of the discrete Hardy-Sobolev inequality as in \cite{BT} and on Moser's iteration method. We point out that our a priori bounds are (in principal) explicit.
\end{abstract}

\maketitle

\section{Introduction}
The purpose of this paper is two-fold. The first aim is motivated by the general principle
that when there exist major results for a semilinear elliptic  boundary value problem,
then if we formulate a reasonable discretization of this boundary value problem
with a view to finding approximate solutions,  then there should exist analogous results
for the corresponding discretized problem. A typical example of this idea may be found in \cite{mcr}.

\medskip

A well-known theorem of Gidas, Ni, and Nirenberg, \cite{GNN},   states roughly that positive
solutions of the semi linear elliptic boundary value problem
\begin{equation}\label{basic_cont}
-\Delta u =f(u), \quad x\in  \Omega, \quad u(x)=0, \quad x\in \partial \Omega
\end{equation}
inherit symmetries of the domain $\Omega$. For example, if  $\Omega$ is a ball,
then all positive solutions must be radially symmetric.  If $\Omega$ is a hypercube,
then all positive solutions must be symmetric about the bisecting hyperplanes.

\medskip

It is natural to ask whether there is a corresponding result for the corresponding
discretized problem. In other words, if we replace the Laplacian in equation \eqref{basic_cont}
with the corresponding finite difference Laplacian ($e_1,\ldots, e_n$ is the standard basis
of $\R^n$ and $h_1,\ldots,h_n>0$ stand for the mesh sizes)
$$
\Delta_h u(x) := \sum_{i=1}^n \frac{u(x+h_ie_i)-2 u(x) +u(x-h_ie_i)}{h_i^2}
$$
we obtain the following finite difference version of \eqref{basic_cont}
\begin{equation}
-\Delta_hu = f(u), \quad x\in  \Omega_h, \quad u(x)=0, \quad x\in \partial \Omega_h.
\label{basic_0}
\end{equation}
The questions is: does a  discrete solution satisfy the same type of symmetries?
(Assuming, of course, that the discretized grid reflects these symmetries.)
The answer is no, even in one dimension. Easy counterexamples of this can be found in \cite{mcr}.
The problem is that there is no restriction on the mesh size.  One should only expect
the (positive) solutions of  equation (2) to reflect those of equation (1)
when the mesh sizes are {\it small}.

\medskip

This gives a clue to the correct result, also in \cite{mcr}, which can be roughly summarized as follows;  {\it as the space step of the discretization
becomes small, the solutions $u_h$ become approximately symmetric. }  Concrete estimates of the distance  from symmetry are given in terms of
the difference between the solution and its reflection about the bisecting hyperplane.
Full details can be found in \cite{mcr}.
This is an example of the general principle mentioned above.  If the discretization is sufficiently fine, the properties of the continuous
solution of (1) should be reflected
in the properties of the discretized solution of (2).

\medskip

Of course, several technical assumptions, both on the Lipschitz constants for $f$
and the behavior of the approximate solutions $u_h$ are required.
One key assumption was that there exists $M>0$ such that  $\| u_h \|_{\infty} \leq M$. Since
this can often be obtained for a wide variety of nonlinearities via the discrete maximum principle,
we felt, at the time, that this assumption was not unreasonable, and indeed natural
in the numerical context. After all, if $\|u_h\|_{\infty} \rightarrow \infty$ as a subsequence
of the $h\rightarrow 0$, one would naturally think that we were not in the neighborhood of a true
classical or weak solution. However, as recently observed in \cite{mcr2} and \cite{hmr}, the blow-up of the $\|\cdot\|_\infty$-norm of a family of finite-difference solutions may indicate the existence of an unbounded distributional solution induced by a supercritical exponent in the nonlinearity, cf. Remark (a) after Theorem~\ref{main}.

\medskip

However, there is another large class of nonlinearities, to which the maximum principle
is not applicable, but for which the true solutions of \eqref{basic_cont} can be shown to satisfy
a priori bounds. A summary of these results are to be found in \cite{mcr2}.
To extend some of the results of \cite{mcr} to a discrete setting is the second aim of this paper.

\medskip

These are the two goals mentioned at the beginning of this introduction. First,
proving these a priori estimates will achieve the aim of proving corresponding results
for solutions of the discretized problems. Second, it will extend to a much wider class
of nonlinearities the approximate symmetry results of \cite{mcr}.

\medskip

A priori estimates will be proven for positive solutions of the following generalization
of \eqref{basic_0}
\begin{equation}
\label{basic}
-\Delta_h u = f(x,u) \mbox{ in } \Omega_h, \quad u=0 \mbox{ on } \partial \Omega_h
\end{equation}
where $\Omega=(a_1,b_1)\times\ldots\times(a_n,b_n)\subset \R^n$ is an $n$-dimensional box and $\Omega_h$ are the finitely many points of a suitable mesh belonging to $\Omega$. 

\begin{theorem}[A priori bounds in dimension $n=1$]
Let $L > 0$, $\Omega=(-L,L)$ and assume that there are continuous functions
$f:\overline{\Omega}\times [0,\infty)\to \R$, $g:[0,\infty)\to \R$ and a constant $K>0$ such that
\renewcommand{\theenumi}{\roman{enumi}}
\begin{enumerate}
\item $g(0)>0$, $g$ is strictly increasing and with $G(s):=\int_{K}^{s} g(t) \; dt$ one has
$$\lim_{s \to \infty} \frac{s}{\sqrt{G(s)}} = 0;$$
\item $f(x,s) \ge g(s) \text{ for all }s \ge K \text{ and for all } x \in \overline{\Omega}$.
\end{enumerate}
Moreover, let $L \ge 4 h$.
Then there exists a constant $\overline{M}$, which is independent of the mesh size $h$,
such that if $u: \overline{\Omega}_{h}\to [0,\infty)$ is a non-negative solution
of \eqref{basic} then $\|u\|_\infty \leq \overline{M}$.
\renewcommand{\theenumi}{\arabic{enumi}}
\label{main_1d}
\end{theorem}

\begin{theorem}[A priori bounds in dimension $n\geq 2$]
\label{main}
Let $n\geq 2$ and let $\Omega=(a_1,b_1)\times\ldots\times(a_n,b_n)\subset\R^n$ be a bounded $n$-dimensional box.
Suppose there exist constants $C_1, C_2, C_3>0$ and $f: \Omega\times [0,\infty)\to [0,\infty)$ such that
\begin{itemize}
\item[(i)] $f(x,s) \geq \lambda s - C_1$ for all $x\in\Omega$ and all $s\geq 0$ and for $\lambda>\lambda_1=\sum_{i=1}^n \frac{\pi^2}{(b_i-a_i)^2}$,
\item[(ii)] $f(x,s) \leq C_2 s^p+C_3$ for some $p\in (1,\frac{n}{n-1})$.
\end{itemize}
Then there exists a constant $\overline{M}$, which is independent of the mesh size $h$,
such that if $u:\overline{\Omega}_h\to [0,\infty)$ is a non-negative solution of \eqref{basic}
then $\|u\|_\infty \leq \overline{M}$.
\end{theorem}

\noindent
{\bf Remarks.} (a) Let us explain the role of the exponent $\frac{n}{n-1}$ from the point of view of differential equations rather then difference equations. As it has been shown in \cite{mcr2}, the exponent $p^\ast=\frac{n}{n-1}$ is a critical exponent for $n$-dimensional boxes in the following sense: if $1< p<p^\ast$ then every very-weak solution (a special form of a distributional solution) of $-\Delta u = u^p$ in $\Omega$ with $u=0$ on $\partial\Omega$ is in fact classical and if it is additionally positive then its $L^\infty$-norm is uniformly bounded. In contrast, for some $p>p^\ast$ one can construct unbounded very-weak solutions of the problem. Hence $p^\ast=\frac{n}{n-1}$ separates (uniformly) bounded solutions from unbounded very-weak solutions. Now, in the finite difference context, we meet the same exponent. We believe that for some $p>p^\ast$ one can construct a family of solutions of finite-difference boundary value problems of the above type where the $L^\infty$-norm of the solution tends to 
infinity as the mesh-size $h$ tends to zero. Some numerical evidence (in a finite-element context, however) is given in \cite{hmr}.

(b) For the one-dimensional as well as for the higher-dimensional case
explicit upper estimates for the value of $\overline{M}$ are immediate from the proofs. Since the formulas are highly complex we decided against writing the details. However, let us point out how such a formula is constructed: at the end of the proof of Theorem~\ref{main} an explicity upper bound for the $L^2$-norm of the discrete gradient of a solution $u$ is obtained, cf. formula \eqref{explicit_bound_D_norm}. If this is inserted into the result of Theorem~\ref{moser_nonlinear} then one obtains an \emph{explicit} upper bound for the $L^\infty$-norm of the solution $u$. 

\medskip

The plan of this paper is to first set up \eqref{basic}, the discretized version of equation
\eqref{basic_cont} in Section 2, and then in Section 3 to describe the discrete function spaces
in which we prove the a priori estimates. Also in Section 3 we state a number of important inequalities (Poincar\'{e}'s inequality, Sobolev's inequality, Hardy's inequality). The proof of theses inequalities is given in the Appendix. In Section 4 we treat the one-dimensional case and give the proof of Theorem~\ref{main_1d}. It is conceptually much different from the treatment of the higher-dimensional case. This is done in Section 5 with the essential ingredient of the Hardy-Sobolev inequality (cf. \cite{BT} where it is used to the same effect of obtaining $L^\infty$ a-priori bounds) and Moser's iteration scheme, which we adapt to the finite-difference context. We conclude with some remarks on some possible extensions and open questions.

\medskip

A final remark before we begin. Throughout the paper we will use the following notation.
For $a\in \R$ let $a_{+}=\max\{a,0\}$ and $a_{-}= \min\{a,0\}$. Let
$e_1,\ldots, e_n$ be the standard basis of $\R^n$ and let
$h := (h_1,\ldots,h_n)>0$ be the mesh size vector of a uniform mesh
$\R^n_h :=\{(h_1z_1,\ldots, h_nz_n): z_1,\ldots, z_n\in \Z\}$.  We use
the short-hand $\delta_i := h_ie_i\in \R^n$ and $\h := h_{1} \cdot \dots \cdot h_{n}$.


\section{Discretization of the domain and the Laplacian}

Consider a (possibly unbounded) $n$-dimensional Euclidean box (also called hypercube)
$$
\Omega = (a_1,b_1)\times\ldots\times(a_n,b_n)\subset \R^n
$$
together with its closure $\overline{\Omega}$ and boundary $\partial\Omega$. We
assume that
$$
a_i = k_i h_i, \quad b_i = l_i h_i \mbox{ with } l_{i} - k_{i} > 1 \mbox{ and } k_i,l_i
\in \Z\cup\{-\infty,\infty\} \mbox{ and } 1 \le i \le n.
$$
Note that we allow $\Omega$ to have unbounded directions. For a box $\Omega$ let us define
discretizations of $\overline{\Omega}$, $\Omega$ and $\partial\Omega$:
$$
\begin{array}{rll}
\overline{\Omega}_h &:= \overline{\Omega} \cap \R^n_h &\mbox{ the set of grid
  points},\vspace{\jot}\\
\Omega_h &:= \Omega \cap \R^n_h &\mbox{ the set of interior grid points},
  \vspace{\jot}\\
\partial \Omega_h &:= \partial\Omega\cap \R^n_h & \mbox{ the set of boundary
  grids points}.
\end{array}
$$
If $\Omega$ is bounded, we also define
  \begin{equation*}
    |\Omega_{h}| := |\Omega| := \prod_{i=1}^{n} (b_{i} - a_{i}).
  \end{equation*}
A finer description of the discrete boundary $\partial\Omega_h$ is given as
follows:
$$
\begin{array}{ll}
\partial_i^+\Omega_h := \partial\Omega_h \cap \{x\in \R^n: x_i=b_i\}
 &\mbox{forward boundary points in direction } e_i, \vspace{\jot}\\
\partial_i^-\Omega_h :=  \partial\Omega_h \cap \{x\in \R^n: x_i=a_i\}
 &\mbox{backward boundary points in direction } e_i.
\end{array}
$$
Let $u:\overline{\Omega}_h\to \R$ be a given function. Our basic concept is
the forward and backward finite difference quotient defined as
$$
\begin{array}{rcll}
D_i^+ u(x) &:=& \displaystyle\frac{u(x+\delta_i)-u(x)}{h_i} & \mbox{ for
} x\in \overline{\Omega}_h\setminus\partial_i^+\Omega_h, \vspace{\jot}\\
D_i^- u(x) &:=& \displaystyle\frac{u(x)-u(x-\delta_i)}{h_i} & \mbox{ for
} x\in \overline{\Omega}_h\setminus\partial_i^-\Omega_h
\end{array}
$$
for $i=1,\ldots,n$.

\begin{definition} Let $u:\overline{\Omega}_h\to\R$. For all $x\in \Omega_h$
the discrete Laplace operator of $u$ at $x$ is given by
\begin{eqnarray}
\Delta_h u(x) &:=& \sum_{i=1}^n D_i^- D_i^+ u(x) = \sum_{i=1}^n D_i^+ D_i^-
u(x) \label{defi_lap}\\
&=& \sum_{i=1}^n \frac{u(x+\delta_i)-2 u(x) +u(x-\delta_i)}{h_i^2}. \nonumber
\end{eqnarray}
\end{definition}

\begin{definition} A function $\phi:\overline{\Omega}_h\to\R$ is said to have compact
support if and only if the set $\supp \phi := \{ x \in \overline{\Omega}_h : \phi(x) \neq 0 \}$ is bounded and
$\supp \phi \subset \Omega_{h}$. In particular, this implies
$\left.\phi \right|_{\partial \Omega_{h}} = 0$.
\end{definition}

\begin{lemma}\label{lm:DISC_WEAK_FORM}
Let $u:\overline{\Omega}_h\to \R$ and $f:\Omega_h\to \R$ be two functions. Then
$$
\sum_{i=1}^n \sum_{x\in \overline{\Omega}_h \setminus \partial_i^+\Omega_h}
D_i^+u(x) D_i^+ \phi(x) \h = \sum_{x\in \Omega_h} f(x) \phi(x)\h
\quad \mbox{ for all }\phi:\overline{\Omega}_h\to\R \mbox{ with compact support}
$$
holds if and only if $-\Delta_h u=f$ in $\Omega_h$.
\end{lemma}

The proof requires the following product rule
\begin{eqnarray}
D_i^-(u\cdot v)(x) &=& v(x) D_i^- u(x)+ u(x-\delta_i) D_i^-v(x) \label{back}\\
&=& v(x) D_i^- u(x)+ u(x-\delta_i) (D_i^+v)(x-\delta_i) \nonumber\\
D_i^+(u\cdot v)(x) &=& v(x) D_i^+ u(x)+ u(x+\delta_i) D_i^+v(x) \label{for}\\
&=& v(x) D_i^+ u(x)+ u(x+\delta_i) (D_i^-v)(x+\delta_i)\nonumber
\end{eqnarray}
and for $w:\overline{\Omega}_h \to \R$ with compact support the summation rule
$$
\sum_{x\in \overline{\Omega}_h \setminus \partial_i^+\Omega_h} D_{i}^{+} w(x) h_i=0, \quad
\sum_{x\in \overline{\Omega}_h \setminus \partial_i^-\Omega_h} D_{i}^{-} w(x) h_i = 0
$$
for each fixed $i=1,\ldots,n$.

\medskip

\noindent
{\it Proof of Lemma~\ref{lm:DISC_WEAK_FORM}:}
We may extend $u$ and $\phi$ to all of $\R^n$ by setting $u=0$, $\phi=0$ outside
$\overline{\Omega}_h$. In this way the value of $D_i^+ u(x)$ is well-defined
everywhere and $D_i^+\phi(x)=0$ for all $x\in \partial_i^+ \Omega_h\cup \overline{\Omega}_h^C$
and all $1 \le i \le n$. We compute
\begin{align*}
\sum_{i=1}^n \sum_{x\in \overline{\Omega}_h\setminus\partial_i^+\Omega_h} D_i^+u(x) D_i^+ \phi(x) \h & = \sum_{i=1}^n \sum_{x\in\R_h^n} D_i^+u(x) D_i^+ \phi(x) \h  \\
&= \sum_{i=1}^n \sum_{x\in \R^n_h} (D_i^+
u)(x-\delta_i) (D_i^+\phi)(x-\delta_i) \h\\
&= \sum_{i=1}^n \sum_{x\in \R^n_h} D_i^-\big(D_i^+
u(x)\phi(x)\big)\h -\big(D_i^- D_i^+ u(x)\big) \phi(x) \h \\
&= \sum_{i=1}^n \sum_{x\in \Omega_h}
-\big(D_i^- D_i^+ u(x)\big) \phi(x)\h,
\end{align*}
where the last equality holds since the first sum $\sum_{x\in \R^n_h} D_i^-\big(D_i^+
u(x)\phi(x)\big)\h$ vanishes and the second summand $\big(D_i^- D_i^+ u(x)\big) \phi(x) \h$ has no contribution outside $\Omega_h$ because $\phi$ vanishes there. According to the
definition (\ref{defi_lap}) of the discrete Laplacian we get the claim.
\qed

\medskip

\noindent
{\bf Remark.} If $U:\Omega\to \R$ is a $C^4$-function then locally
$$
\Delta_h U(x) = \Delta U(x) + O(h^2)
$$
for $x\in \Omega_h$.


\section{Discrete function spaces, inequalities and embeddings}

We can consider spaces of functions $u:\overline{\Omega}_h \to \R$
defined on (possibly unbounded) domains $\Omega_h$. Let $1\leq p<\infty$ and define $A(t) =
e^{t^2}-1$ for $t\in \R$. We define the following norms
\begin{align*}
\|u\|_{L^p} &=\Big(\sum_{x\in\overline{\Omega}_h}
|u(x)|^p \h\Big)^{1/p},\quad \|u\|_{L^\infty} = \sup_{x\in
  \overline{\Omega}_h} |u(x)|,\\
\|u\|_D &= \Big(\sum_{i=1}^n \sum_{x\in
  \overline{\Omega}_h\setminus\partial_i^+\Omega_h}|D_i^+ u(x)|^2\h\Big)^\frac{1}{2},\\
\|u\|_{W^{1,2}} &= (\|u\|_{L^{2}}^{2} + \|u\|_{D}^{2})^{\frac{1}{2}} =
\Big(\sum_{i=1}^n
\sum_{x\in\overline{\Omega}_h\setminus\partial_i^+\Omega_h} |D_i^+ u(x)|^2\h +
\sum_{x\in\overline{\Omega}_h} |u(x)|^2 \h \Big)^\frac{1}{2},\\
\|u\|_A &= \inf\Big\{k>0: \sum_{x\in\overline{\Omega}_h}
A\Big(\frac{|u(x)|}{k}\Big)\h \leq 1\Big\}.
\end{align*}

\noindent
{\bf Remark.} For $u\not =0$ the infimum in the definition of $\|u\|_A$ is a minimum (a consequence of Fatou's lemma). Moreover, for $u\not =0$ the statement $\|u\|_A\leq t$ is equivalent to 
$\sum_{x\in\overline{\Omega}_h} A\Big(\frac{|u(x)|}{t}\Big)\h \leq 1$.

\medskip

Corresponding to these norms we define the function spaces
\begin{align*}
L^p(\overline{\Omega}_h) &:= \{u: \overline{\Omega}_h \to \R:
\|u\|_{L^p}<\infty\}, \\
L^\infty(\overline{\Omega}_h) &:= \{u:
\overline{\Omega}_h\to \R: \|u\|_{L^\infty}<\infty\},\\
W^{1,2}(\overline{\Omega}_h) &:= \{u: \overline{\Omega}_h \to \R:
\|u\|_{W^{1,2}}<\infty\}, \\
L^A(\overline{\Omega}_h) &:= \{ u: \overline{\Omega}_h \to \R: \|u\|_A<\infty\}.
\end{align*}
Moreover, we define 
$$ 
W_0^{1,2}(\overline{\Omega}_h) :=
\closure_{W^{1,2}}\{u: \overline{\Omega}_h \to \R:
u \mbox{ has compact support in } \overline{\Omega}_h\}
$$
where the closure is taken with respect to the $W^{1,2}$-norm. Finally, for $n \ge 3$ let
  \begin{equation*}
    D^{1,2}(\R^n_h) := \{u: \R^n_h \to \R:
    \|u\|_{L^{\frac{2n}{n-2}}} < \infty, \|u\|_{D}<\infty\}.
  \end{equation*}

\begin{lemma} The spaces
  $(L^\infty(\overline{\Omega}_h),\|\cdot\|_{L^\infty})$,
  $(L^A(\overline{\Omega}_h), \|\cdot\|_A)$,
  $(L^p(\overline{\Omega}_h),\|\cdot\|_{L^p})$ for $1\leq p < \infty$,
  $(W^{1,2}(\overline{\Omega}_h),\|\cdot\|_{W^{1,2}})$, $(W_0^{1,2}(\overline{\Omega}_h),\|\cdot\|_{W^{1,2}})$ as well as the space $(D^{1,2}(\R^n_h,\|\cdot\|_D)$ are Banach spaces. All five norms satisfy the inequality $\|u_{+}\|, \|u_{-}\|\leq
  \|u\|$. If $\Omega_h= \R^n_h$ then functions with compact support are dense in all these Banach spaces except for $L^\infty(\R^n_h)$. Elements of $W_0^{1,2}(\overline{\Omega}_h)$ vanish on $\partial\Omega_h$ and if $\Omega_h$ is bounded then they also have compact support.
\end{lemma}

\begin{proof} The only nontrivial statement is the completeness of $D^{1,2}(\R^n_h)$. Let $(u_k)_{k\in \N}$ be a Cauchy sequence in $D^{1,2}(\R^n_h)$. By the Sobolev inequality of Theorem~\ref{sobolev_embedd} the sequence is also a Cauchy sequence in $L^\frac{2n}{n-2}(\R^n_h)$ and hence convergent, i.e. $u_k\to u\in L^\frac{2n}{n-2}(\R^n_h)$ as $k\to \infty$. Moreover, from the definition of $\|\cdot\|_D$ it follows that for each $i=1,\ldots,n$ the sequence $(D_i^+ u_k)_{k\in\N}$ is a Cauchy sequence in $L^2(\R^n_h)$ and hence convergent, i.e., there exist functions $f_i\in L^2(\R^n_h)$ such that $D_i^+ u_k \to f_i$ as $k\to \infty$ in the sense of $L^2$-convergence. The proof is complete if we can show that $D_i^+ u = f_i$. Since for a subsequence $(u_{k_l})_{l\in \N}$ we have $u_{k_l}(x)\to u(x)$ as $l\to \infty$ for almost all $x\in \R^n_h$ the definition of the forward finite-difference quotient implies that $D_i^+ u_{k_l}(x) \to D_i^+ u(x)$ as $l\to \infty$ for almost all $x\in \R^n_h$ and for $i=1,\ldots, n$. This implies $D_i^+ u = f_i$ and finishes the proof.  
\end{proof}

Next we give statements of four different discrete inequalities:
Poincar\'{e}'s inequality in Theorem~\ref{poincare},
a Sobolev inequality in dimension $\geq 3$ in Theorem~\ref{sobolev_embedd},
a Sobolev inequality in dimension 2 in Theorem~\ref{sobolev_embedd_2d},
and Hardy's inequality in Theorem~\ref{th:HARDY_INEQUALITY}.
All four inequalities have continuous counterparts. For completeness the proofs,
which are suitable variants of the proofs of the continuous counterparts,
are given in the Appendix. Additionally, by combining the Sobolev inequality with the Hardy inequality we obtain the so-called Hardy-Sobolev inequality of Theorem~\ref{th:HARDY_SOBOLEV}, where the proof is given directly after the statement. We begin with the following Lemma.

\begin{lemma}[Properties of the first eigenfunction]
Let $\Omega=(a_1,b_1)\times\ldots\times (a_n,b_n)\subset\R^n$ be a bounded $n$-dimensional box.
\begin{itemize}
\item[(i)] The first eigenvalue of the discrete Laplace-operator $(-\Delta_h)$
with vanishing Dirichlet boundary conditions on $\partial\Omega_h$
is simple and given by
$$
\lambda_{1,h} = \sum_{i=1}^n \frac{4}{h_i^2}\sin^2\left(\frac{\pi h_i}{2(b_i-a_i)}\right)
              < \sum_{i=1}^n \frac{\pi^2}{(b_i-a_i)^2} =: \lambda_1
$$
with the corresponding eigenfunction
$$
\phi_{1,h}(x) = \prod_{i=1}^n \sin\frac{\pi(x_i-a_i)}{b_i-a_i} > 0,
\quad x \in \Omega_h.
$$
\item[(ii)] Let $t>0$ be a fixed value such that
$t \|u\|_{L^{1}} = \sum_{x\in \Omega_h} t\phi_{1,h}(x)\h =1$. Then
$$
t \phi_{1,h}(x) \geq \frac{2^n}{|\Omega|^2} \dist(x,\partial\Omega_h)^n \mbox{ for all } x \in \Omega_h.
$$
\end{itemize}
\label{normalization}
\end{lemma}

\begin{proof} (i): By direct computation one verifies that
  $u_{s}(z)= \sin \tfrac{\pi s(z-a)}{b-a}$,  $1 \le s \le l_{1} - k_{1} -1$,
  form a complete set of discrete eigenfunctions of $-\Delta_{h}$ on the discrete
  $1$-dimensional box  $(a,b) \cap \R_{h}$ with $a=k_{1}h$, $b=l_{1}h$, $l_{1},k_{1} \in \Z$,
  $l_{1} - k_{1} > 1 \in \Z$ with eigenvalues
    \begin{equation*}
      \lambda_{s} = \frac{4}{h^{2}} \sin^{2} \frac{s \pi h}{2 (b-a)}.
    \end{equation*}
  On a bounded $n$-dimensional box the operator $-\Delta_{h}$ with vanishing Dirichlet boundary conditions is the sum of the corresponding
  one-dimensional operators. Hence we can construct all eigenfunctions for
  the $n$-dimensional operator by taking tensor products of the eigenfunctions
  of the corresponding one-dimensional operators. The eigenvalues are then given
  as the sum of corresponding eigenvalues. In particular, this implies that
  $\lambda_{1,h}$ is simple. The estimate $\lambda_{1,h}<\lambda_1$ follows from $\sin x <x $ for $x>0$.

\medskip

(ii): We begin with a simple observation: if $f:\Omega_h\to \R$ is a function of the form
$f(x)=\prod_{i=1}^n f_i(x_i)$ then
$$
\sum_{x\in \Omega_h} f(x)\h = \prod_{i=1}^n \left(\sum_{k=1}^{N_i-1} f_i(a_i+k h_i) h_i\right),
$$
where $N_i := l_{i} - k_{i}$. Based on this we can estimate the discrete $L^1$-norm of $\phi_{1,h}$:
\begin{align*}
  \sum_{x\in\Omega_h} \phi_{1,h}(x)\h
  & \leq \left(\sum_{x\in \Omega_h} \phi_{1,h}(x)^2\h\right)^\frac{1}{2} |\Omega_h|^\frac{1}{2} \\
  & = \prod_{i=1}^n \left( h_{i} \sum_{k=1}^{N_i-1}
  \underbrace{\sin^2 \frac{\pi k h_i}{b_i-a_i}}_{\leq 1}\right)^\frac{1}{2} |\Omega|^\frac{1}{2} \\
& \leq \prod_{i=1}^n (h_i N_i)^\frac{1}{2} |\Omega|^\frac{1}{2}= |\Omega|
\end{align*}
and hence $t \geq 1/|\Omega|$. Next we use the concavity estimate $\sin x\geq \frac{2x}{\pi}$
for $x\in [0,\frac{\pi}{2}]$ and derive from it
$$
\sin \frac{\pi(x_i-a_i)}{b_i-a_i} \geq \frac{2(x_i-a_i)}{b_i-a_i} \quad \mbox{ provided } a_i \leq x_i \leq \frac{a_i+b_i}{2}.
$$
Since $\dist(x, \partial \Omega_{h}) = \min_{1 \le i \le n} \{ x_{i} - a_{i}, b_{i} - x_{i}\}$
in $\Omega_{h}$, we obtain
$$
 \phi_{1,h}(x) \geq \frac{2^n \dist(x,\partial\Omega_h)^n}{|\Omega|} \quad \mbox{ for all } x \in\Omega_h
$$
which finishes the proof of the claim.
\end{proof}

\medskip

\noindent
{\bf Remark.} Note the trivial statement: $u\in L^p(\overline{\Omega}_h) \Rightarrow u\in L^\infty(\overline{\Omega}_h)$
with the embedding estimate $\|u\|_{L^\infty} \leq \h^{-1/p}\|u\|_{L^p}$,
which is not stable with respect to the mesh size $h$.
In contrast, the following embeddings will be stable with respect to $h$.

\begin{theorem}[Poincar\'{e}'s inequality]
Let $\Omega=(a_1,b_1)\times\ldots\times (a_n,b_n)\subset\R^n$ be a bounded $n$-dimensional box.
Then there exists a constant $C_P(\Omega)$ which is independent of $h$ such that
\begin{equation} \label{eq:POINCARE_INEQ}
\|u\|_{L^2} \leq C_P(\Omega)\|u\|_D\; \mbox{  for all }
u \in W_0^{1,2}(\overline{\Omega}_h).
\end{equation}
Moreover, we have
  \begin{equation*}
    \frac{1}{\lambda_{1,h}} \le  C^{2}_P(\Omega)
                                   \le \frac{1}{4n^{2}} \sum_{i=1}^n (b_i-a_i)^2,
  \end{equation*}
where the lower bound is optimal.
\label{poincare}
\end{theorem}

\begin{theorem}[Sobolev embedding for $n\geq 3$] Let $n\geq 3$. With
  $C_S(n):=\frac{4(n-1)}{\sqrt{n}(n-2)} \le 5$ the following inequality holds
$$
\|u\|_{L^\frac{2n}{n-2}} \leq C_S(n) \|u\|_D \mbox{ for all } u \in
D^{1,2}(\R^n_h),
$$
and for every mesh size vector $h$.
\label{sobolev_embedd}
\end{theorem}

\begin{theorem}[Sobolev embedding for $n=2$] With $C_S(2) := 8\sqrt{2\pi(e+256)}$ the
  following inequality holds
\begin{equation}
\|u\|_A \leq C_S(2) \|u\|_{W^{1,2}} \mbox{ for all } u \in W^{1,2}(\R^2_h),
\label{sobolev_2d}
\end{equation}
and for every the mesh size vector $h$. This is equivalent to
\begin{equation}
\sum_{x\in \R^2_h} A\Big(\frac{|u(x)|}{C_S(2)\|u\|_{W^{1,2}}}\Big)\h \leq 1
\mbox{ for all } u \in W^{1,2}(\R^2_h)\setminus\{0\},
\label{sobolev_2d_explicit}
\end{equation}
and in particular it implies
\begin{equation}
\|u\|_{L^p}\leq 2C_S(2)p \|u\|_{W^{1,2}} \mbox{ for all } u \in
W^{1,2}(\R^2_h) \mbox{ and all } p\geq 2.
\label{power-sobolev_2d}
\end{equation}
\label{sobolev_embedd_2d}
\end{theorem}

Recall Hardy's inequality for a bounded Lipschitz-domain $\Omega$, cf. \cite{brezis_marcus}, \cite{mat_sobo}:
$$
\int_\Omega \frac{u^2}{\dist(x,\partial\Omega)}\,dx \leq C_H\int_\Omega
|\nabla u|^2\,dx \qquad \mbox{for all } u\in W_0^{1,2}(\Omega),
$$
where $C_H\in (0,1/4)$ is a constant depending only on $\Omega$. For convex
domains $\Omega$ it is known that $C_H=1/4$. For an $n$-dimensional box we give next a discrete analogue of Hardy's inequality.

\begin{theorem}[Hardy's inequality]\label{th:HARDY_INEQUALITY}
 Let $\Omega=(a_1,b_1)\times\ldots\times (a_n,b_n)\subset\R^n$ be a bounded $n$-dimensional box.
There exists a constant $C_H\in (0,4]$ such that
$$
\sum_{x\in \Omega_h} \frac{u(x)^2}{\dist(x,\partial\Omega_h)^2}\h \leq C_H
\sum_{i=1}^n \sum_{x\in \overline{\Omega}_{h}\setminus\partial_{i}^{+} \Omega_h} |D_i^+u(x)|^2\h
$$
for all $u\in W^{1,2}_{0}(\overline{\Omega}_h)$ and for every mesh size vector $h$.
\end{theorem}

\begin{theorem}[Hardy-Sobolev inequality]\label{th:HARDY_SOBOLEV}
Let $\Omega=(a_1,b_1)\times\ldots\times (a_n,b_n)\subset \R^n$ be a bounded $n$-dimen\-sional box and let $\alpha, \beta\geq 0$ be two numbers such that
$$
0 \leq \beta <2, \quad \beta \leq \alpha
\left\{
\begin{array}{ll} \leq \frac{2n-2\beta}{n-2} & \mbox{ if } n \geq 3, \vspace{\jot} \\
< \infty & \mbox{ if } n=2.
\end{array}\right.
$$
Then the following Hardy-Sobolev inequality holds true
$$
\sum_{x\in\Omega_h} \frac{|u(x)|^\alpha}{\dist(x,\partial\Omega_h)^\beta}\h
\leq C_{HS}(n,\alpha,\beta,\Omega) \|u\|_D^\alpha
\quad \mbox{ for all } u \in W_0^{1,2}(\overline{\Omega}_h),
$$
where the constant $C_{HS}(n,\alpha,\beta,\Omega)$ is given by
$$
C_{HS}(n,\alpha,\beta,\Omega) = \left\{
\begin{array}{ll}
C_H^{\beta/2} C_S(n)^{\alpha-\beta}|\Omega|^\frac{2(n-\beta)-(n-2)\alpha}{2n} & \mbox{ if } n\geq 3, \vspace{\jot}\\
C_H^{\beta/2} C_S(2)^{\alpha-\beta}\left(\frac{4(\alpha-\beta)}{2-\beta}\right)^{\alpha-\beta} (1+C_P(\Omega)^2)^\frac{\alpha-\beta}{2} & \mbox{ if } n=2.
\end{array}
\right.
$$
\end{theorem}

\begin{proof} First assume $n\geq 3$. By using a triple H\"older-inequality and Hardy's inequality
we obtain for $u\in W_0^{1,2}(\overline{\Omega}_h)$
\begin{align*}
\sum_{x\in\Omega_h} \frac{|u(x)|^\alpha}{\dist(x,\partial\Omega_h)^\beta}\h & = \sum_{x\in\Omega_h} \frac{|u(x)|^\beta}{\dist(x,\partial\Omega_h)^\beta} |u|^{\alpha-\beta}\h \\
& \leq \left(\sum_{x\in\Omega_h} \frac{|u(x)|^2}{\dist(x,\partial\Omega_h)^2}\h  \right)^\frac{\beta}{2} \left(\sum_{x\in\Omega_h} |u|^\frac{2n}{n-2}\h\right)^\frac{(n-2)(\alpha-\beta)}{2n} \left(\sum_{x\in\Omega_h}\h\right)^ \frac{2(n-\beta)-(n-2)\alpha}{2n}\\
& \leq C_H^{\beta/2} C_S(n)^{\alpha-\beta} |\Omega|^\frac{2(n-\beta)-(n-2)\alpha}{2n} \|u\|_D^\alpha.
\end{align*}
In the case $n=2$ we use Hardy's inequality and the two-dimensional Sobolev inequality \eqref{power-sobolev_2d} from Lemma~\ref{sobolev_embedd_2d}
\begin{align*}
\sum_{x\in\Omega_h} \frac{|u(x)|^\alpha}{\dist(x,\partial\Omega_h)^\beta}\h & = \sum_{x\in\Omega_h} \frac{|u(x)|^\beta}{\dist(x,\partial\Omega_h)^\beta} |u|^{\alpha-\beta}\h \\
& \leq \left(\sum_{x\in\Omega_h} \frac{|u(x)|^2}{\dist(x,\partial\Omega_h)^2}\h  \right)^\frac{\beta}{2} \left(\sum_{x\in\Omega_h} |u|^\frac{2(\alpha-\beta)}{2-\beta}\h\right)^\frac{2-\beta}{2} \\
& \leq C_H^{\beta/2} \left(C_S(2) \frac{4(\alpha-\beta)}{2-\beta}\right)^{\alpha-\beta} \|u\|_D^\beta \|u\|_{W^{1,2}}^{\alpha-\beta} \\
& \leq C_H^{\beta/2}  C_S(2)^{\alpha-\beta} \left(\frac{4(\alpha-\beta)}{2-\beta}\right)^{\alpha-\beta} (1+C_P(\Omega)^2)^\frac{\alpha-\beta}{2} \|u\|_D^\alpha,
\end{align*}
where we have used the inequality $\|\cdot\|_{W^{1,2}}^2\leq(1+C_P(\Omega)^2)\|\cdot\|_D^2$.
\end{proof}


\section{One-dimensional case}

For one-dimensional case we introduce the following notation
    \begin{equation*}
      (a,b)_{h} := (a,b) \cap \R_{h}, \quad [a,b]_{h} := [a,b] \cap \R_{h}, etc.
    \end{equation*}

\begin{lemma}[Discrete Elliptic Comparison]\label{lm:DISC_ELL_COMP}
  Let $a,b,h\in \R$ with $h>0$ such that $\frac{b-a}{h} \in \N \setminus \{ 1 \} $.
Suppose that \map{w}{[a,b]_{h}}{\R} satisfies
  \begin{equation}\label{eq:ELLP_COMP_INEQ}
    -D^{+}_{h} D^{-}_{h} w(x) \le \lambda_{0,h} w(x),
    \quad x \in (a,b)_{h}, \quad w(a), w(b)\leq 0
\end{equation}
with some $\lambda_{0,h} < \lambda_{1,h}$, where
$\lambda_{1,h} = \tfrac{4}{h^{2}} \sin^{2} \left( \tfrac{\pi h}{2(b-a)} \right)$
is the first Dirichlet-eigenvalue of the discrete one-dimensional Laplacian on $[a,b]_h$
as given in Lemma~\ref{normalization}. Then, $w(x) \le 0$ for all $x \in [a,b]_{h}$.
\end{lemma}

\begin{proof}
  We multiply (\ref{eq:ELLP_COMP_INEQ}) with $w_{+}$, sum up over $(a,b]_{h}$
  and exploit partial summation like in Lemma~\ref{lm:DISC_WEAK_FORM} to obtain
    \begin{align*}
      \lambda_{0,h} \sum_{x \in (a,b)_{h}}  w_{+}^{2}(x) & \ge \sum_{x \in [a,b)_{h}} D^{+}_{h} w(x) \cdot D^{+}_{h} w_{+}(x) \\
      & \ge \sum_{x \in [a,b)_{h}} D^{+}_{h} w_{+}(x) \cdot D^{+}_{h} w_{+}(x).
    \end{align*}
  To see the last inequality let $x \in [a,b)_{h}$.
  If $w(x),w(x+h) \ge 0$ or $w(x),w(x+h) \le 0$ then
    \begin{equation*}
      D^{+}_{h} w(x) D^{+}_{h} w_{+}(x) = D^{+}_{h} w_{+}(x) D^{+}_{h} w_{+}(x).
    \end{equation*}
  If $w(x) \le 0 \le w(x+h)$ then $D^{+}_{h} w(x) \ge D^{+}_{h} w_{+}(x)\geq 0$, implying
    \begin{equation*}
      D^{+}_{h} w(x) D^{+}_{h} w_{+}(x) \ge D^{+}_{h} w_{+}(x) D^{+}_{h} w_{+}(x).
    \end{equation*}
  Finally, if $w(x) \ge 0 \ge w(x+h)$ then $D^{+}_{h} w(x) \le D^{+}_{h} w_{+}(x)\leq 0$, also implying
    \begin{equation*}
      D^{+}_{h} w(x) D^{+}_{h} w_{+}(x) \ge D^{+}_{h} w_{+}(x) D^{+}_{h} w_{+}(x).
    \end{equation*}
Using the variational characterization of $\lambda_{1,h}$ from Theorem \ref{poincare} we get
    \begin{equation*}
      0 \le \sum_{x \in [a,b)_{h}} \left( D^{+}_{h} w_{+}(x) \right)^{2}
      \le \lambda_{0,h} \sum_{x \in (a,b)_{h}}  w_{+}^{2}(x)
      \le \frac{\lambda_{0,h}}{\lambda_{1,h}}
      \sum_{x \in [a,b)_{h}} \left( D^{+}_{h} w_{+}(x) \right)^{2},
    \end{equation*}
  and since $w_{+}(a)=0$ consequently $w_{+} \equiv 0$.
\end{proof}

\begin{lemma}[Poisson problem] Let $a,b,h\in \R$ with $h>0$ such that
$\frac{b-a}{h} \in \N \setminus \{ 1 \} $ and consider the Poisson problem
$$
-D_h^+ D_h^- v(x)= \mu v(x) -A, \quad x \in (a,b)_{h}, \quad v(a)=\gamma, v(b)=0
$$
with $\gamma, A \ge 0$ and $\mu=\frac{4}{h^2}\sin^{2} \left( \tfrac{\pi h}{4(b-a)} \right)$. Then the unique solution is given by
\begin{equation}
v(x) = \left(\gamma-\frac{A}{\mu}\right)\cos\left(\frac{\pi (x-a)}{2(b-a)}\right) - \frac{A}{\mu} \sin\left(\frac{\pi(x-a)}{2(b-a)}\right)+\frac{A}{\mu}
\label{formel}
\end{equation}
and satisfies
$$
-D_h^+ v(a) \leq \left(\gamma+(b-a)^2A\right)\frac{\pi}{2(b-a)}.
$$
\label{poisson}
\end{lemma}

\begin{proof}
If we compute the second order finite difference quotient of $\sin(\nu k h)$ then one finds
\begin{align*}
D_h^+ D_h^- \sin(\nu k h)
&= \frac{1}{h^2}\left( \sin(\nu kh + kh)- 2\sin(\nu k h)+\sin(\nu k h-\nu h)\right )\\
&= \frac{2}{h^2} \sin(\nu k h) (\cos(\nu h)-1)
= -\frac{4}{h^2} \sin^2\left(\frac{\nu h}{2}\right) \sin(\nu k h),
\quad k\in \N.
\end{align*}
The same equality holds if $\sin(\nu k h)$ is everywhere replaced by $\cos(\nu k h)$.
With $\nu = \frac{\pi}{2(b-a)}$ we see that $w$ as given in \eqref{formel}
satisfies the Poisson problem with $\frac{4}{h^2} \sin^2\left(\frac{\nu h}{2}\right) = \mu$. Since $\mu<\lambda_{1,h}$ (the first Dirichlet eigenvalue of the one-dimensional Laplacian on $(a,b)_h$) uniqueness follows from the comparison principle of Lemma~\ref{lm:DISC_ELL_COMP}. Finally, let us compute (using $\sin x \leq x, 1-\cos x\leq x$ for $x\in (0,\pi/2)$):
\begin{align*}
-D_h^+ v(a) &= \frac{1}{h}\left( \gamma-\frac{A}{\mu}\right)
               \left(1-\cos\left(\frac{\pi h}{2(b-a)}\right)\right)
            + \frac{A}{\mu h}\sin\left(\frac{\pi h}{2(b-a)}\right) \\
            &\le  \frac{\gamma}{h}
               \left(1-\cos\left(\frac{\pi h}{2(b-a)}\right)\right)
            + \frac{A}{\mu h}\sin\left(\frac{\pi h}{2(b-a)}\right)
            \leq \left(\gamma+\frac{A}{\mu}\right) \frac{\pi}{2(b-a)}.
\end{align*}
Since $\sin x \geq 2x/\pi$ for $x\in (0,\pi/2)$ we see that
$$
\mu = \frac{4}{h^2} \sin^2\left(\frac{\nu h}{2}\right) \geq \frac{1}{(b-a)^2}
$$
which together with the previous estimate yields the result.
\end{proof}

Before we give the proof of Theorem~\ref{main_1d} we have to derive estimates
from the hypotheses on $f$ and $g$.

\begin{lemma} Let the assumptions of Theorem~\ref{main_1d} hold and let $\lambda_0=\left(\frac{\pi}{L}\right)^2$.
\begin{itemize}
\item[(a)] There exist $A>0$ such that $f(x,s)\geq \lambda_0 s-A$ for all $s\geq 0$ and all $x\in[-L,L]$.
\item[(b)] For any $R>K$ we have
\begin{equation}
\int_0^R \frac{ds}{\sqrt{G(R)-G(s)}} \leq \frac{2R}{\sqrt{G(R)}}.
\label{eq:1D_AUX}
\end{equation}
\end{itemize}
\label{prelim}
\end{lemma}

\begin{proof} (a): From (i) in Theorem~\ref{main_1d} we have $G(s) \le (s-K) g(s)$,
  yielding
     \begin{equation*}
      \frac{s^{2}}{G(s)} \ge \frac{s^2}{(s-K)g(s)}
      \ge \frac{s}{g(s)}
    \end{equation*}
  for $s > K > 0$. By assumption (ii) of Theorem~\ref{main_1d}, the l.h.s. of this inequality
  converges to zero and since the r.h.s. is positive, we have
    \begin{equation*}
      \lim_{s \to \infty}\frac{s}{g(s)} = 0
    \end{equation*}
  which implies that $g$ grows faster than any linear function. In particular, there exists $K_1\geq K$ such that
  \begin{equation*}
      g(s) \ge \lambda_{0} s \mbox{ for all }s \ge K_{1}.
    \end{equation*}
  Since $g$ is continuous,
  we can set
    \begin{equation*}
      A_{g} := \max_{s \in [0,K_{1}]} \left( \lambda_{0} s - g(s) \right)_{+} \ge 0,
    \end{equation*}
  which implies
    \begin{equation*}
      g(s) \ge \lambda_{0} s - A_{g} \mbox{ for all }s \ge 0.
    \end{equation*}
  Since $f$ is continuous on $[-L,L] \times [0,K_{1}]$,
  we repeat the argument, setting
    \begin{equation*}
    \begin{aligned}
      A_{f} &:= \max_{\substack{s \in [0,K_{1}]\\ x \in [-L,L]}}
      \left( \lambda_{0} s - f(x,s) \right)_{+} \ge 0 .
    \end{aligned}
    \end{equation*}
  Taking into account (iii) of Theorem~\ref{main_1d} we have analogously
    \begin{equation}\label{eq:1D_LAMBDA}
      f(x,s) \ge \lambda_{0} s - A \mbox{ for all } s \ge 0 \mbox{ and } x \in [-L,L]
    \end{equation}
  with $A := \max\{A_{f},A_{g}\}$.

\medskip

\noindent
(b): Since $g$ is positive and strictly increasing we get for every $t \in (0,1)$
  \begin{equation*}
  \begin{aligned}
    G(Rt) = \int_{K}^{Rt} g(s) d s = t \int_{K/t}^{R} g(t \tau) d \tau
    \le t \int_{K}^{R} g(t \tau) d \tau
    \le t \int_{K}^{R} g(\tau) d \tau = t G(R)>0.
  \end{aligned}
  \end{equation*}
This yields
  $$
     \frac{1}{\sqrt{1 - \frac{G(Rt)}{G(R)}}} \le \frac{1}{\sqrt{1-t}}
  $$
for $t\in (0,1)$ and we get
  \begin{equation*}
  \begin{aligned}
    \int_{0}^{R} \frac{d s}{\sqrt{G(R) - G(s)}}
    &= R \int_{0}^{1} \frac{d t}{\sqrt{G(R) - G(Rt)}}
    = \frac{R}{\sqrt{G(R)}} \int_{0}^{1} \frac{d t}{\sqrt{1 - \frac{G(Rt)}{G(R)}}}\\
    &\le \frac{R}{\sqrt{G(R)}} \int_{0}^{1}  \frac{d t}{\sqrt{1-t}}
    = \frac{2 R}{\sqrt{G(R)}}.
  \end{aligned}
  \end{equation*}
\end{proof}

\noindent
\emph{Proof of Theorem~\ref{main_1d}:} Let $u$ be an arbitrary but fixed non-negative solution of (\ref{basic}). We denote $M:=\|u\|_{\infty}$. Since our problem is symmetrical w.r.t axis reflection, we can assume
that $M= u(x_{0})$ with $x_{0} \le 0$.
Both $R := u(x_{0} + h)$ and $u(x_{0} +2 h)$ are well-defined
due to $L \ge 4h$ and from Lemma~\ref{prelim}(a) we obtain
  \begin{equation*}
    \frac{- M + 2R + 0}{h^{2}} \ge \frac{- u(x_{0}) + 2 u(x_{0}+h) - u(x_{0} + 2 h)}{h^{2}}
    = f(x_0+h,R) \ge \lambda_{0} R - A \ge -A
  \end{equation*}
  so that
  $$
    M \le 2 R + A h^{2} \le 2 R + A \frac{L^{2}}{16}.
  $$
It is therefore sufficient to find a bound for $R$ and without loss of generality we assume $R > K$. Since $0 < K<R\leq M$, there exists $x_{1} \in [x_{0}, L-h]_{h}$ such that
$u \ge K$ on $[x_{0},x_{1}]_{h}$ and $u(x_{1} + h) < K$.
Since $R>K$, we have $x_{1} \neq x_{0}$, i.e.
$[x_{0} + h,x_{1}] \neq \emptyset$.
Using (i) and (ii) of Theorem~\ref{main_1d} we see that as long as $x \in [x_{0},x_{1}]_{h}$
  \begin{equation*}
    D^{+}_{h} D^{-}_{h} u(x) = -f(x,u(x)) \le - g(u(x)) < 0
  \end{equation*}
which means that $D^{-}_{h} u(x+h)<  D^{-}_{h} u(x)$ for $x \in [x_{0},x_{1}]_{h}$.
It follows inductively that
  \begin{equation*}
    D^{+}_{h} u(x) = D^{-}_{h} u(x+h) < D^{-}_{h} u(x)
    < \ldots < D^{-}_{h} u(x_{0} + h) = D^{+}_h u(x_{0}) = \frac{R - M}{h} \le 0
  \end{equation*}
and hence $u(x+h) < u(x)$ as long as $x \in [x_{0}+h,x_{1}]_{h}$.

\medskip

We now want to derive an upper bound for the forward difference of $u$
at $x_{1}$. For $x \in [x_{0}+h,x_{1}]_{h}$ we have on one hand, using
(ii) of Theorem~\ref{main_1d}
  \begin{equation*}
  \begin{aligned}
    - D^{+}_{h} \left((D^{-}_{h} u(x))^{2} \right) &= - D^{+}_{h} (D^{-}_{h}u(x) \cdot D^{-}_{h}u(x)) \\
    &=- D^{-}_{h}u(x+h) \cdot D^{+}_{h} D^{-}_{h} u(x) - D^{-}_{h}u(x) \cdot D^{+}_{h}D^{-}_{h} u(x)\\
    &=\underbrace{-D^{+}_{h} D^{-}_{h} u(x)}_{\ge g(u(x)) > 0}
    (\underbrace{D^{-}_{h} u(x+h)}_{< 0} + \underbrace{D^{-}_{h}u(x)}_{\le 0})\\
    &\le g(u(x)) D^{-}_{h} u(x+h) = g(u(x)) D^{+}_{h} u(x)
  \end{aligned}
  \end{equation*}
and on the other hand with some $\bar{u} \in [u(x+h),u(x)]$
  \begin{equation*}
    D_{h}^{+}G(u(x)) =\frac{1}{h} \int_{u(x)}^{u(x+h)} g(t) d t= g(\bar{u}) \frac{u(x+h)-u(x)}{h}
    \ge g(u(x)) D^{+}_{h} u(x).
  \end{equation*}
Together, we obtain
  \begin{equation*}
    - D^{+}_{h} \left( (D^{-}_{h} u(x))^{2} \right) \le D_{h}^{+}G(u(x))
  \end{equation*}
on $x \in [x_{0}+h,x_{1}]_{h}$ and summing up over $[x_{0}+h,x]_{h}$, $x \le x_{1}$
we also get
  \begin{equation*}
  \begin{aligned}
    - \left( D^{-}_{h} u(x + h) \right)^{2}
    &\le
    - \left( D^{-}_{h} u(x + h) \right)^{2}
    + \left( D^{-}_{h} u(x_{0}+h) \right)^{2}\\
    &= - \sum_{z \in [x_{0}+h,x]_{h} } D^{+}_{h}
         \left[ \left( D^{-}_{h} u(z) \right)^{2} \right] h \\
    &\le \sum_{z \in [x_{0}+h,x]_{h} } D^{+}_{h} G(u(z)) h\\
    &= G(u(x+h)) - G(u(x_{0}+h)).
  \end{aligned}
  \end{equation*}
  This shows that
$$
(\underbrace{D^{+}_{h} u(x)}_{< 0})^{2} \ge \underbrace{G(R) - G(u(x + h))}_{> 0}
$$
and finally for $x \in [x_{0}+h,x_{1}]_{h}$
  \begin{equation}\label{eq:1D_UPPER_ESTIMATE}
    D_{h}^{+} u(x) \le - \sqrt{G(R) - G(u(x+h))} < 0.
  \end{equation}
In particular, this inequality holds for $x=x_1$.

\medskip

We now want to show that $[x_{0},x_{1}]$ shrinks as $R$ goes to infinity.
More precisely, we want to show that
  \begin{equation*}
    x_{1} - x_{0} \le \frac{2R}{\sqrt{G(R)}}.
  \end{equation*}
Consider
 $$
    \kappa(t) := \int_{t}^{R} \frac{d t}{\sqrt{G(R) - G(s)}}  d s, \quad t \in [0,R]
 $$
 with $\kappa$ being well-defined due to (\ref{eq:1D_AUX}).
 Since $G(t)$ is increasing on $[0,R)$, we see that the function
 $\kappa'(t)= -\frac{1}{\sqrt{G(R) - G(t)}}$ is negative and decreasing in $t\in [0,R)$.
 Gathering our results we see that for $x \in [x_{0}+h, x_{1}]_{h}$
  \begin{equation*}
  \begin{aligned}
    D^{+}_{h} \kappa(u(x)) &= \frac{1}{h} \int_{u(x)}^{u(x+h)} \kappa'(t) d t
    = \frac{1}{h} \int_{u(x+h)}^{u(x)} -\kappa'(t) d t\\
    &\ge - \kappa'(u(x+h)) \cdot (- D^{+}_{h} u(x))
     = \frac{-1}{\sqrt{G(R) - G(u(x+h))}} D^{+}_{h} u(x)
  \end{aligned}
  \end{equation*}
and by using (\ref{eq:1D_UPPER_ESTIMATE}) this results in
  \begin{equation*}
  \begin{aligned}
    D_{h}^{+} \kappa(u(x)) \ge 1 \mbox{ for } x \in [x_{0}+h, x_{1}]_{h}.
  \end{aligned}
  \end{equation*}
Using this and (\ref{eq:1D_AUX}) we obtain
  \begin{equation*}
  \begin{aligned}
    x_{1} - x_{0} &\le \sum_{x \in [x_{0} + h, x_{1}]_{h}} D_{h}^{+} \kappa(u(x)) h
    =  \kappa(u(x_{1} + h )) - \kappa(R)\\
    &= \int_{u(x_{1} + h)}^{R} \frac{1}{\sqrt{G(R) - G(s)}}  d s \le \int_{0}^{R} \frac{1}{\sqrt{G(R) - G(s)}}  d s
    \le \frac{2R}{\sqrt{G(R)}}.
  \end{aligned}
  \end{equation*}
Because of assumption (ii) of Theorem~\ref{main_1d} we can define
  \begin{equation*}
    R_{1} := \min \left\{ r > K \mid
                                \frac{2 \rho}{\sqrt{G(\rho)}} \le \frac{L}{2}, \quad
                                \forall \rho \ge r \right\}.
  \end{equation*}
Assuming now w.l.o.g. that $R \ge R_{1}$, we have
  \begin{equation*}
   x_1 \leq x_{1} - x_{0} \le \frac{L}{2}, \quad \mbox{ i.e. } \quad L - x_{1} \ge \frac{L}{2}.
  \end{equation*}
Let \map{v}{[x_{1},L]}{\R} be the solution of
  $$
    -D^{+}_{h} D^{-}_{h} v(x) = \mu v(x) - A \;\mbox{ for }\;x \in (x_{1},L)_{h}, \qquad v(x_{1}) = u(x_{1}), \quad v(L) = 0
  $$
  with $\mu=\frac{4}{h^2}\sin^2\left(\frac{\pi h}{4(L-x_1)}\right)$. The solution and its properties are given in Lemma~\ref{poisson}. Since $u(x) \ge 0$ and $\mu \leq \pi^4/(4(L-x_1)^2)\leq (\pi/L)^2=\lambda_0$ we see that $u$ satisfies
  $$
    -D^{+}_{h} D^{-}_{h} u(x) = f(x,u(x)) \ge  \lambda_{0} u(x) - A \ge  \mu u(x) - A
                                 \;\mbox{ for }\; x \in (x_{1},L)_{h}.
  $$
Since $\mu$ is smaller than the first Dirichlet eigenvalue of the one-dimensional Laplacian on the interval $[x_1,L]$ we can apply the comparison principle of Lemma \ref{lm:DISC_ELL_COMP} for $w:=v - u$ and get $v(x) \le u(x)$ on $[x_{1},L]_{h}$
and due to $v(x_{1})=u(x_{1})$ also $D^{+}_{h}v(x_{1}) \le D^{+}_{h} u(x_{1})$.

\medskip

Using (\ref{eq:1D_UPPER_ESTIMATE}) and
$u(x_{1}) \le R$ we get
  \begin{align*}
  \sqrt{G(R) - G(K)} & \le \sqrt{G(R) - G(u(x_{1}+h))}\\
    &\le -D_h^+ u(x_1) \\
    & \leq -D_h^+ v(x_1) \\
    & \leq \left(u(x_1)+(L-x_1)^2A\right)\frac{\pi}{2(L-x_1)} \\
    & \leq R \frac{\pi}{L} + \frac{A\pi L}{2}
  \end{align*}
implying
$$
    \frac{G(R)}{R^{2}} \le \frac{G(K)}{R^{2}}
     + \left( \frac{\pi}{L} +  \frac{A \pi L}{2R} \right)^{2} \le  \frac{G(K)}{K^{2}}
     + \left( \frac{\pi}{L} +  \frac{A \pi L}{2 K} \right)^{2}
$$
Since $G(t)/t^{2} \to \infty$ as $t \to \infty$, we have the following bound
  \begin{equation*}
    R \le \min \left\{ r >K \mid
    \frac{G(\rho )}{\rho^{2}} > \frac{G(K)}{K^{2}}
     + \left( \frac{\pi}{L} +  \frac{A \pi L}{2 K} \right)^{2}
     \quad \forall \rho > r\right\}.
  \end{equation*}
\qed

\section{Higher dimensional case}

On the level of linear equations in higher dimensions the basic step for $L^\infty$
a priori bounds is done via the Moser iteration scheme (see Chapter 8.5 in \cite{GT}).
It turns out that this method also works in discrete settings.

\begin{theorem}[$L^\infty$-bounds for linear equations]
Let $\Omega\subset \R^n$ be a bounded $n$-dimensional box and let $u\in W_0^{1,2}(\overline{\Omega}_h)$ satisfy
\begin{equation}
-\Delta_h u = a(x)u + b(x) \mbox{ in } \Omega_h, \quad u=0 \mbox{ on }
\partial\Omega_h.
\label{sub}
\end{equation}
Then for $q>\frac{n}{2}$
\begin{equation}
\|u\|_\infty \leq (2K)^{q'}\|u\|_{L^2}+ 2K\|b\|_\infty(1+ |\Omega|),
\label{infinity}
\end{equation}
with
\begin{align*}
K &=
\max \left\{\frac{n}{(n-2)q'}, C_S(n)\sqrt{2}\bigl\||a|+1\bigr\|_{L^q}^{1/2}\right\}
^{\left(\frac{n(q-1)}{2q-n}\right)^2}
 & \mbox{ if } n \geq 3,\\
K & =
\max \left\{ \frac{\hat{n}}{(\hat{n}-2) q'},
\frac{4 \hat{n}}{(\hat{n}-2)}  C_S(2)(1+C_P(\Omega))\sqrt{2}\bigl\||a|+1\bigr\|_{L^q}^{1/2}
\right\}^{\left(\frac{{\hat n}(q-1)}{2q-\hat n}\right)^2} & \mbox{ if } n =2,
\end{align*}
where $C_P(\Omega)$ is the Poincar\'{e} constant of $\Omega_h$, $C_S(n)$ are the
respective Sobolev constants and in the case $n=2$ the value $\hat n$ can be
chosen arbitrary such that $q>\frac{\hat n}{2}>1$.
\label{moser}
\end{theorem}

\begin{proof}
The proof relies on the following basic inequality
  $1\leq \frac{1-\tau^q}{1-\tau}\leq q$ for every $q\geq 1$ and
  $\tau\in [0,1)$. A simple consequence of this inequality is
\begin{equation}
\frac{(y^{2s+1}-z^{2s+1})(y-z)}{(y^{s+1}-z^{s+1})^2} \geq \frac{1}{s+1} \geq
\frac{1}{(s+1)^2} \mbox{ for all } s,y,z \geq 0.
\label{simp}
\end{equation}
Set $\phi:=u_+^{2s+1}\ge 0$ with $s \ge 0$. Since $\phi$
has compact support in $\overline{\Omega}_{h}$ the weak formulation of (\ref{sub}) implies
$$
\sum_{i=1}^n \sum_{x\in \overline{\Omega}_h \setminus \partial_i^{+} \Omega_h}
D_i^+ u(x) D_i^+ \phi(x) \h =\sum_{x\in\Omega_h} \Bigl(a(x) u(x)+ b(x)\Bigr)\phi(x)\h.
$$
Let $i \in \{ 1, \dots, n\}$ and $x \in \overline{\Omega}_h \setminus \partial_i^{+} \Omega_h$
be arbitrary but fixed, then
  \begin{equation*}
    D_i^{+} u(x) D_i^{+} \phi(x) \ge D_i^{+} u_{+}(x) D_i^{+} \phi(x).
  \end{equation*}
To see this, consider the following four possibilities:
if $u(x),u(x+\delta_{i}) \ge 0$ then $D_i^{+} u(x) = D_i^{+} u_{+}(x)$,
if $u(x),u(x+\delta_{i}) \le  0$ then $D_i^{+} \phi(x) = 0$,
if $u(x) \le 0 \le u(x+\delta_{i})$ then $D^{+}_{i} u(x) \ge D^{+}_{i} u_{+}(x)$
and $D_i^{+} \phi(x) \ge 0$, and finally if $u(x) \ge  0 \ge u(x+\delta_{i})$ then
$D^{+}_{i} u(x) \le D^{+}_{i} u_{+}(x)$ and $D_i^{+} \phi(x) \le 0$.
In all four cases the inequality follows. Moreover, from (\ref{simp}) we have
$$
\frac{1}{(s+1)^2} \big(D_i^{+} u_+^{s+1}(x)\big)^2
\le D_i^{+} u_+(x) D_i^{+} u_+^{2s+1}(x)
= D_i^{+} u_{+}(x) D_i^{+}\phi(x).
$$
This yields
\begin{align}
\frac{1}{(s+1)^2} \|u_+^{s+1}\|_D^2
&\leq \sum_{x\in\Omega_h} \Bigl(a(x) u(x)+b(x)\Bigr) \phi(x) \h = \sum_{x\in\Omega_h} \Bigl(a(x) u_{+}(x)+b(x)\Bigr) u_+(x)^{2s+1} \h  \nonumber \\
&\leq \sum_{x\in\Omega_h} \Bigl(|a(x)| u_{+}(x) + |b(x)|\Bigr) u_+(x)^{2s+1}\h \nonumber \\
& \leq \sum_{x\in\Omega_h} \Bigl( (|a(x)|+1) u_{+}(x)^{2s+2} + |b(x)|^{2s+2}\Bigr)\h   \label{iter_0}\\
& \leq  \bigl\||a|+1\bigr\|_{L^q}(\|u_+^{2s+2}\|_{L^{q'}}+ \bigl\||b|^{2s+2}\bigr\|_{L^{q'}})\nonumber 
\end{align}
For $n\geq 3$ we may use Sobolev's inequality with constant $C_S(n)$. If we define the constant $C=C_S(n)\bigl\||a|+1\bigr\|_{L^q}^{1/2}$ then \eqref{iter_0} leads to
\begin{equation}
\|u_+\|_{L^{\frac{2n}{n-2}(s+1)}} \leq
\Big((s+1)C\Big)^{1/(s+1)}  \left(\|u_+\|_{L^{2q'(s+1)}}^{2s+2}+\|b\|_{L^{2q'(s+1)}}^{2s+2}\right)^{1/(2s+2)}.
\label{iter_1}
\end{equation}
In the case $n=2$ we will obtain (\ref{iter_1}) with a different
constant $C$ and with $n$ replaced by $\hat{n}$
as explained at the end of the proof.
Next we set $s+1=t$ and define recursively
$$
t_k\frac{2n}{n-2} = 2q' t_{k+1}, \quad t_0=1,
$$
i.e., $t_k=l^k$ with $l=\frac{n}{(n-2)q'} >1$ for $q > \frac{n}{2}$.
Furthermore, set
$m_k=\|u_+\|_{L^{2q't_k}}$. Then (\ref{iter_1}) is equivalent to
\begin{equation}
m_{k+1}\leq (t_k C)^{1/t_k}(m_k^{2t_k}+\|b\|_{L^{2q't_k}}^{2t_k})^{1/(2t_k)}, \quad m_0=\|u_+\|_{L^{2q'}}.
\label{rec}
\end{equation}
We compare the sequence $(m_k)_{k\in \N_0}$ with the sequence $(\tilde m_k)_{k\in \N_0}$ which we suppose to satisfy
$$
\tilde m_{k+1} = (t_k \sqrt{2}C)^{1/t_k}\tilde m_k, \quad \tilde m_0=\|u_+\|_{L^{2q'}}+ \|b\|_\infty (1+|\Omega|).
$$
Let us explain that $\tilde m_k \geq m_k$ for all $k\in \N_0$. It is clearly true for $k=0$. Moreover, since $\tilde m_k$ is increasing in $k$, we have 
$$
\tilde m_k \geq \tilde m_0 \geq \|b\|_\infty (1+|\Omega|) \geq \|b\|_\infty \|\Omega|^{1/s} \geq \|b\|_{L^s} \mbox{ for } s\geq 1
$$
and hence
\begin{equation}
\tilde m_{k+1}\geq (t_k C)^{1/t_k}(\tilde m_k^{2t_k}+\|b\|_{L^{2q't_k}}^{2t_k})^{1/(2t_k)}.
\label{rec_tilde}
\end{equation}
Comparing \eqref{rec_tilde} with \eqref{rec} and using $\tilde m_0\geq m_0$ we find by induction that $\tilde m_k \geq m_k$ for all $k\in \N_0$. Hence 
$$
\|u_+\|_\infty = \lim_{k\to \infty} m_k \leq \lim_{k\to \infty} \tilde m_k = \prod_{i=0}^\infty (t_i \sqrt{2}C)^{1/t_i} \tilde m_0
= \tilde m_0 \exp\left(\sum_{i=0}^\infty \frac{1}{t_i}\log(t_i\sqrt{2}C)\right).
$$
Since
$$
\sum_{i=0}^\infty \frac{1}{t_i}\log(t_i\sqrt{2}C)
= \sum_{i=0}^\infty \left(\frac{1}{l^i}\log \sqrt{2}C + \frac{i}{l^i}\log l\right)
\leq \log \tilde{C} \sum_{i=0}^\infty \frac{i+1}{l^i} = \log \tilde{C} \left( \frac{l}{l-1}\right)^2,
$$
where $\tilde{C} := \max \{\sqrt{2} C, l \}$, this implies the estimate
\begin{equation}
\| u_+\|_\infty \leq \underbrace{\tilde{C}^{\left(\frac{n(q-1)}{2q-n}\right)^2}}_{=:K} \tilde m_0 = K\bigl(\|u_+\|_{L^{2q'}}+\|b\|_\infty (1+|\Omega|)\bigr).
\label{almost_done}
\end{equation}
Adopting the proof from Gilbarg-Trudinger~\cite{GT}, Chapter 7.1
we get the discrete interpolation inequality $\|u_+\|_{L^\beta} \leq
\epsilon\|u_+\|_{L^\gamma} + \epsilon^{-\mu} \|u_+\|_{L^\alpha}$ for
$1\leq \alpha\leq \beta\leq \gamma\leq\infty$ and $\mu =
\Big(\frac{1}{\alpha}-\frac{1}{\beta}\Big)/\Big(\frac{1}{\beta}-\frac{1}{\gamma}\Big)$.
If we apply this to (\ref{almost_done}) with $\alpha=2$,
$\beta=2q'$, $\gamma=\infty$ and $\epsilon = 1/(2K)$ then we obtain
 $$
    \|u_+\|_\infty \leq \frac{1}{2} \|u_+\|_\infty + (2K)^{q'-1}K \|u_+\|_{L^2}+K\|b\|_\infty (1+|\Omega|),
$$
i.e.,
$$
    \|u_{+}\|_{\infty}  \le (2K)^{q'} \|u\|_{L^2}+2K\|b\|_\infty (1+|\Omega|),
$$
which implies the desired inequality (\ref{infinity}) for $u_{+}$.
The inequality for $u_{-}$ is derived in a similar way by replacing $u$
with $-u$ and $b$ with $-b$.

\medskip

Now we come back to the case $n=2$. By Poincar\'{e}'s inequality
we have the estimate $\|\cdot\|_{W^{1,2}}^2 \leq (1+C_P(\Omega)^2)\|\cdot\|_D^2$. Thus from
(\ref{iter_0}) we get
$$
\|u_+^{s+1}\|_{W^{1,2}}^2 \leq (1+C_P(\Omega))^2(s+1)^2\bigl\||a|+1\bigr\|_{L^q}
\left(\|u_+^{2s+2}\|_{L^{q'}}+\bigl\||b|^{2s+2}\bigr\|_{L^{q'}}\right).
$$
Using the 2-dimensional Sobolev inequality from Theorem \ref{sobolev_embedd_2d}
and choosing $\hat{n}$ such that $q>\frac{\hat n}{2}>1$,
i.e. $\frac{\hat{n}}{(\hat{n} -2) q'} > 1$ we get
$$
\|{\bar u}^{s+1}\|_{L^{\frac{2\hat n}{\hat n-2}}}^2
\Big[2C_S(2)\frac{2\hat n}{\hat n-2}\Big]^{-2} \leq (1+C_P(\Omega))^2(s+1)^2\bigl\||a|+1\bigr\|_{L^q}
\left(\|u_+^{2s+2}\|_{L^{q'}}+\bigl\||b|^{2s+2}\bigr\|_{L^{q'}}\right).
$$
Thus we may proceed with (\ref{iter_1}) where the constant $C$ now
takes the value
$$
C=\frac{4\hat n}{\hat n-2}C_S(2)(1+C_P(\Omega))\bigl\||a|+1\bigr\|_{L^q}^{1/2}.
$$
\end{proof}

Just like in the theory for linear elliptic continuous boundary value problems the above $L^\infty$ bound for linear difference equations can be transferred to subcritical nonlinear difference equations  as shown next.

\begin{theorem}[$L^\infty$-bounds for nonlinear equations]
Let $\Omega\subset \R^n$ be a bounded $n$-dimensional box. Assume
  $1\leq p<\frac{n+2}{n-2}$ for $n \ge 3$ and
  $1\leq p< \infty$ for $n = 2$. For $n\geq 3$ let
  $q=\frac{2n}{(n-2)(p-1)}$ and for $n=2$ choose
  $q>\max\{1,\frac{2}{p-1}\}$. Let $u\in W_0^{1,2}(\overline{\Omega}_h)$
  satisfy
\begin{equation}
-\Delta_h u = f(x,u) \mbox{ in } \Omega_h, \quad u=0 \mbox{ on }
\partial\Omega_h,
\label{eq_non}
\end{equation}
where $f:\Omega_h\times \R\to \R$ satisfies $|f(x,s)|\leq C_2 |s|^p + C_3$ for all $s\in \R$, $x\in\Omega_h$ and some constant $C_2, C_3>0$. Then
\begin{equation}
\|u\|_\infty \leq (2K_{u})^{q'}C_P(\Omega)\|u\|_D+ 2K_u C_3(1+|\Omega|),
\label{infinity_non}
\end{equation}
with
$$
K_{u} =\max \left\{\frac{n}{(n-2)q'}, C_S(n)\sqrt{2}\sqrt{C_2 C_S(n)^{p-1}\|u\|_D^{p-1}+|\Omega|^\frac{1}{q}}   \right\}^{\left(\frac{n(q-1)}{2q-n}\right)^2}
$$
if $n \geq 3$
and
$$
K_{u}  = \max\left\{\frac{\hat{n}}{(\hat{n}-2)q'}, \frac{4{\hat n}}{(\hat n-2)} \bigl(2C_S(2)(1+C_P(\Omega))q(p-1)\bigr)^\frac{p+1}{2}\sqrt{C_2\|u\|_D^{p-1}+|\Omega|^\frac{1}{q}}
\right\}^{\left(\frac{{\hat n}(q-1)}{2q-\hat n}\right)^2}
$$
if $n =2$, where $C_P(\Omega)$ is the Poincar\'{e} constant of $\Omega_h$, $C_S(n)$ are the
respective Sobolev constants. In the case $n=2$
the value $\hat n$ can be chosen arbitrary such that $q>\frac{\hat n}{2}>1$.
\label{moser_nonlinear}
\end{theorem}

\begin{proof} Let $u$ be a solution of \eqref{eq_non}. If we define the functions
$$
\tilde{a}(x) := \frac{f(x,u(x))}{C_2 |u(x)|^p+ C_3}, \quad
a(x) := \tilde{a}(x) C_2 |u(x)|^{p-1}\sign u(x), \quad
b(x) := \tilde{a}(x) C_3 \quad \mbox{ for } x\in \Omega_h
$$
then $u$ satisfies the linear equation
$$
-\Delta_h u = a(x)u(x)+ b(x) \quad \mbox{ in } \Omega_h, \quad u=0 \mbox{ on } \partial\Omega_h.
$$
For $n\geq 3$ we have $q(p-1)=\frac{2n}{n-2}$ and hence we find
$$
\|\tilde a\|_\infty \leq 1, \quad \||a|+1\|_{L^q} \leq C_2 \|u\|_{\frac{2n}{n-2}}^{p-1}+|\Omega|^{1/q} \leq C_2 C_S(n)^{p-1} \|u\|_D^{p-1}+|\Omega|^{1/q} \quad \mbox{ and } \quad \|b\|_\infty \leq C_3.
$$
Now the claim follows from Theorem \ref{moser}. In the case $n=2$ we
get with Theorem \ref{sobolev_embedd_2d}
\begin{align*}
\||a|+1\|_{L^q} &\leq C_2\left(2C_S(2)
q(p-1)\right)^{p-1}\|u\|_{W^{1,2}}^{p-1}+ |\Omega|^{1/q} \\ 
& \leq \bigl(\underbrace{2C_S(2)q(p-1)(1+C_P(\Omega))}_{\geq 1}\bigr)^{p-1}\left(C_2\|u\|_D^{p-1}+|\Omega|^{1/q}\right).
\end{align*}
Again the claim follows from Theorem \ref{moser}.
\end{proof}

\noindent
{\it Proof of Theorem \ref{main}:} Let $\phi_{1,h}$ be a first Dirichlet eigenfunction of the  discrete Laplacian on $\Omega_h$. According to Lemma~\ref{normalization} we may normalize $\phi_{1,h}$ such that $\sum_{x\in\Omega_h} \phi_{1,h}\h=1$ and
\begin{equation}
\phi_{1,h}(x) \geq \frac{2^n}{|\Omega|^2} \dist(x,\partial\Omega_h)^n \mbox{ in } \Omega_h.
\label{distance_estimate}
\end{equation}
Testing (\ref{basic}) with $\phi_{1,h} \ge 0$ and using the hypothesis (ii) yields
$$
\lambda_{1,h} \sum_{x\in \Omega} u\phi_{1,h}\h =
\sum_{x\in\Omega} f(x,u)\phi_{1,h}\h
\geq -C_1 + \lambda \sum_{x\in\Omega_h} u \phi_{1,h}\h.
$$
Since by assumption $\lambda>\lambda_1>\lambda_{1,h}$ we obtain
\begin{equation}
\label{bound}
\sum_{x\in\Omega_h} u\phi_{1,h}\h \leq \frac{C_1}{\lambda-\lambda_1} \mbox{ and } \sum_{x\in\Omega_h} f(x,u)\phi_{1,h}\h \leq \frac{C_1\lambda_1}{\lambda-\lambda_1}.
\end{equation}
By assumption (ii) we have $1<p<\frac{n}{n-1}$. Let us define $\tilde p =\frac{1}{2}(p+\frac{n}{n-1})$ so that $p<\tilde p < \frac{n}{n-1}$. Next we test (\ref{basic}) with the solution $u$ itself and find
\begin{align}
\sum_{i=1}^n \sum_{x\in \overline{\Omega}_h\setminus \partial_i^+\Omega_h} |D_i^+u|^2(x)\h  &= \sum_{x\in\Omega_h} f(x,u)u\h \nonumber \\
&= \sum_{x\in\Omega_h} \frac{u}{\phi_{1,h}^{1/{\tilde p}'}} f(x,u) \phi_{1,h}^{1/{\tilde p}'}\h \label{est1}\\
& \leq \Big(\sum_{x\in \Omega_h} \frac{u^{\tilde p}f(x,u)}{\phi_{1,h}^{\tilde p-1}}\h
\Big)^{1/\tilde p} \Big(\sum_{x\in\Omega_h}  f(x,u)\phi_{1,h}\h \Big)^{1/{\tilde p}'}, \nonumber
\end{align}
where ${\tilde p}'$ is the conjugate exponent to $\tilde p$. The last sum in (\ref{est1}) is bounded due to (\ref{bound}). By hypothesis (ii) one obtains
$$
\sum_{x\in \Omega_h} \frac{f(x,u)u^{\tilde p}}{\phi_{1,h}^{\tilde p-1}} \h
\leq \sum_{x\in\Omega_h} \frac{C_2u^{\tilde p+p}+C_3u^{\tilde p}}{\phi_{1,h}^{\tilde p-1}}\h.
$$
Using \eqref{distance_estimate} we obtain
\begin{equation}
\label{est2}
\sum_{x\in\Omega_h} \frac{f(x,u)u^{\tilde p}}{\phi_{1,h}^{\tilde p-1}}\h
\leq \frac{|\Omega|^{2(\tilde p-1)}}{2^{n(\tilde p-1)}}
     \sum_{x\in\Omega_h} \frac{C_2u^{\tilde p+p}+C_3u^{\tilde p}}{\dist(x,\partial\Omega_h)^{n(\tilde p-1)}}\h.
\end{equation}
Next we apply the Hardy-Sobolev inequality from Theorem \ref{th:HARDY_SOBOLEV}
with $\alpha_1 = \tilde p+p$, $\beta=n(\tilde p-1)$ and with $\alpha_2=\tilde p$,
$\beta=n(\tilde p-1)$, respectively. Since $\tilde p<\frac{n}{n-1}$ we see that $\beta=n(\tilde p-1)<\tilde p\leq \alpha_{1,2}$ in both cases.  The condition $\beta<2$ amounts to $n(\tilde p-1)<2$, i.e. $\tilde p < 1+\frac{2}{n}$, which is true since
$\tilde p < \frac{n}{n-1}\leq 1+\frac{2}{n}$ for all $n\geq 2$. Finally, we need to check the condition $\alpha_{1,2}\leq \frac{2}{n-2}(n-\beta)$ for $n>2$. It is enough to check it for $\alpha_1=p+\tilde p$, where it amounts to
$$
p+\tilde p \leq \frac{2n}{n-2}(2-\tilde p), \mbox{ i.e. } \quad p+\tilde p\left(1+\frac{2n}{n-2}\right)\leq \frac{4n}{n-2}.
$$
This inequality is true, since $p, \tilde p < \frac{n}{n-1}$. Hence the
Hardy-Sobolev inequality, (\ref{est1}) and (\ref{est2}) lead to
$$
\|u\|_D^{2\tilde p} \leq \left(\frac{C_1\lambda_1 |\Omega|^2}{(\lambda-\lambda_1)2^n}\right)^{\tilde p-1}\left(
C_2 C_{HS}(n,\alpha_1,\beta,\Omega)\|u\|_D^{p+\tilde p} + C_3C_{HS}(n,\alpha_2,\beta,\Omega)\|u\|_D^{\tilde p}\right)
$$
and therefore
\begin{equation}
\|u\|_D^{\tilde p-p} \leq \max\left\{1,\left(\frac{C_1\lambda_1 |\Omega|^2}{(\lambda-\lambda_1)2^n}\right)^{\tilde p-1}\bigl(C_2 C_{HS}(n,\alpha_1,\beta,\Omega)+ C_3C_{HS}(n,\alpha_2,\beta,\Omega)\bigr) \right\}
\label{explicit_bound_D_norm}
\end{equation}
for every non-negative solution $u$ of (\ref{basic}). Now Theorem~\ref{moser_nonlinear}
applies and shows that $\|u\|_\infty$ is uniformly bounded for every
non-negative solution $u$ of (\ref{basic}). \qed


\section{Open problems and extensions}

Let us finish our discussion with a list of open questions:
\begin{itemize}
\item[(i)] Can one extend Theorem~\ref{main} to more general domains? As a start in this direction one might consider domains which are unions of $n$-dimensional boxes, e.g. an $L$-shaped domain in the case $n=2$. The main difficulty is to find a proof for the statement 
$$
\phi_{1,h}(x) \geq C(\Omega) \dist(x,\partial\Omega_h)^n \mbox{ in } \Omega_h
$$
which appeared as \eqref{distance_estimate} in the proof of Theorem~\ref{main}. Here $\phi_{1,h}$ is the first Dirichlet eigenfunction of $-\Delta_h$ on $\Omega_h$ normalized by $\sum_{x\in \Omega_h} \phi_{1,h}(x) \h=1$.
\item[(ii)]  Can one extend Theorem~\ref{main} to solutions of a finite element version of \eqref{basic}? Here the main difficulty is to find an extension of Theorem~\ref{moser} to finite element solutions. Recall that the proof of Theorem~\ref{moser} is based on Moser's iteration scheme which uses $u_+^{2s+1}$ as a test function for values of $s$ tending to $\infty$. In the finite element setting $u_+^{2s+1}$ is not in the finite element space.
\item[(iii)] What are the optimal constants in the discrete Sobolev inequalities of Theorem~\ref{sobolev_embedd}, Theorem~\ref{sobolev_embedd_2d} and the discrete Hardy inequality of Theorem~\ref{th:HARDY_INEQUALITY}? Are these constants attained?
\item[(iv)] Can one show that for $p>n/(n-1)$ and $f(x,s)=s^p$ there is a sequence of positive finite difference solutions of \eqref{basic}, whose $L^\infty$-norm blows up as the mesh-size goes to zero? This would show that Theorem~\ref{main} is sharp with respect to the exponent. Numerical evidence for the existence of such solutions near a rectangular corner is given in \cite{hmr} in a finite element context.
\end{itemize}

\section*{Appendix}

\noindent
\emph{Proof of Theorem~ \ref{poincare}:} The operator $-\Delta_{h}$
  is a positive self-adjoint operator on the subspace $W^{1,2}_{0} (\overline{\Omega}_{h})$
  of the finite-dimensional Hilbert space $L^{2}(\overline{\Omega}_{h})$. The first eigenvalue $\lambda_{1,h}$ is described in Lemma~\ref{normalization}.
  Therefore, the optimal (smallest) value of $C_P(\Omega)$ in \eqref{eq:POINCARE_INEQ} is coming from Rayleigh's characterization of the smallest eigenvalue, i.e.,
$$
\frac{1}{C_P(\Omega)^2} \le \lambda_{1,h} = \sum_{i=1}^n \frac{4}{h_i^2}\sin^2\left(\frac{\pi h_i}{2(b_i-a_i)}\right).
$$
This is, however, an $h$-dependent quantity. Using $\sin x\geq \frac{2}{\pi}x$ for $0\leq x\leq \pi/2$ we obtain
$$
\lambda_{1,h} \geq \sum_{i=1}^n \frac{4}{(b_i-a_i)^2} \geq n^2 \frac{1}{\sum_{i=1}^n \frac{(b_i-a_i)^2}{4}}
$$
by the harmonic-arithmetic mean inequality. This implies $C_P(\Omega) \leq \frac{1}{2n} \sqrt{\sum_{i=1}^n (b_i-a_i)^2}$.
\qed

\medskip

\noindent
\emph{Proof of Theorem~\ref{sobolev_embedd}:}
We may assume that $u$ has compact support. Then for each $i=1,\ldots,n$ one
has
$$
u(x) = \sum_{k = -\infty, k\in \Z}^{-1} D_i^+ u(x+k\delta_i)h_i
$$
and hence
$$
|u(x)|^\frac{n}{n-1} \leq \prod_{i=1}^n \Big(\sum_{k_i\in \Z}
|D_i^+ u(x+k_i\delta_i)|h_i\Big)^\frac{1}{n-1}.
$$
Further summation yields
\begin{align*}
\sum_{k_1 \in \Z}|u(x+k_1\delta_1)|^\frac{n}{n-1}h_1
\leq& \Big(\sum_{k_1\in \Z}|D_1^+
u(x+k_1\delta_1)|h_1\Big)^\frac{1}{n-1} \\
& \cdot\sum_{k_1\in \Z} \prod_{i=2}^n\Big(\sum_{k_i\in \Z}
|D_i^+ u(x+k_{1} \delta_{1} + k_i\delta_i)|h_1h_i\Big)^\frac{1}{n-1}\\
\leq &\Big(\sum_{k_1\in \Z} |D_1^+
u(x+k_1\delta_1)|h_1\Big)^\frac{1}{n-1}\\
&\cdot\prod_{i=2}^n\Big( \sum_{k_1,k_i \in \Z}|D_i^+
u(x+ k_1\delta_1+k_i\delta_i)|h_1h_i\Big)^\frac{1}{n-1},
\end{align*}
where the last inequality is a consequence of H\"older's inequality
applied to the product of $n-1$-functions inside
$\sum_{k_1\in \Z}$. Summing next over $k_2\in \Z$ we
get
\begin{align*}
\sum_{k_1,k_2\in \Z}^\infty&
|u(x+k_1\delta_1+k_2\delta_2)|^\frac{n}{n-1}h_1h_2\\
\leq &\Big(\sum_{k_1,k_2\in \Z} |D_1^+
u(x+k_1\delta_1+k_2\delta_2)|h_1h_2\Big)^\frac{1}{n-1} \\
&\cdot \Big(\sum_{k_1,k_2 \in \Z}|D_2^+
u(x+k_1\delta_1+k_2\delta_2)|h_1h_2\Big)^\frac{1}{n-1}\\
&\cdot\prod_{i=3}^n\Big( \sum_{k_1,k_2,k_i \in \Z}^\infty |D_i^+
u(x+k_1\delta_1+k_2\delta_2+k_i\delta_i)|h_1h_2h_i\Big)^\frac{1}{n-1},
\end{align*}
so that after $n-2$ further steps one arrives at
$$
\sum_{x\in \R^n_{h}} |u(x)|^\frac{n}{n-1}\h \leq \prod_{i=1}^n
\Big(\sum_{x\in\R^n_h}|D_i^+ u(x)|\h\Big)^\frac{1}{n-1}.
$$
Using the inequality between the geometric and the arithmetic mean we arrive at
\begin{equation}
\label{sobolev_1}
\|u\|_{L^\frac{n}{n-1}} \leq \prod_{i=1}^n \Big(\sum_{x\in\R^n_h}|D_i^+ u(x)|
\h\Big)^\frac{1}{n}
\leq \frac{1}{n} \sum_{i=1}^n \sum_{x\in\R^n_h}|D_i^+ u(x)|\h.
\end{equation}
Next we set $u=|v|^{\gamma-1} v$ with $\gamma=\frac{2n-2}{n-2}$, use the
estimate
$$
|D_i^+ u(x)| \leq \gamma (|v(x)|^{\gamma-1}+|v(x+\delta_i)|^{\gamma-1})
|D_i^+v(x)|
$$
and insert in (\ref{sobolev_1})
\begin{align*}
(\|v\|_{L^\frac{2n}{n-2}})^\frac{2n-2}{n-2}=
\|\,|v|^{\gamma-1} v\,\|_{L^\frac{n}{n-1}} &\leq \frac{\gamma}{n} \sum_{i=1}^n
\sum_{x\in \R^n_{h}} \big(|v(x)|^{\gamma-1}+|v(x+\delta_i)|^{\gamma-1}\big) |D_i^+
v(x)|\h \\
&\leq \frac{\gamma}{n} \Big(\sum_{i=1}^n\sum_{x\in \R^n_h}
4|v(x)|^\frac{2n}{n-2}\h\Big)^\frac{1}{2}\Big(\sum_{i=1}^n \sum_{x\in \R^{n}_{h}}
|D_i^+ v(x)|^2 \h\Big)^\frac{1}{2}.\\
& \leq \frac{2\gamma}{\sqrt{n}} (\|v\|_{L^\frac{2n}{n-2}})^\frac{n}{n-2}
\|v\|_D.
\end{align*}
This finishes the proof of the Sobolev inequality.
\qed

\medskip

The following results of Lemma~\ref{ineq}, Lemma~\ref{additivity}, Lemma~\ref{interpol} and Lemma~\ref{classical_estimate} prepare the proof of the 2-dimensional Sobolev inequality of Theorem~\ref{sobolev_embedd_2d}.

\begin{lemma} Let $p\in \N$ and $a,b,c\geq 0$. Then
$$
(p+2)(p+1)(a^p+b^p+c^p)\geq \frac{1}{b-c}\Big(\frac{b^{p+2}-a^{p+2}}{b-a} -
\frac{c^{p+2}-a^{p+2}}{c-a}\Big) \geq a^p+b^p+c^p.
$$
\label{ineq}
\end{lemma}

\begin{proof}
From
  \begin{equation*}
    \frac{1}{b-c}\Big(\frac{b^{p+2}-a^{p+2}}{b-a} - \frac{c^{p+2}-a^{p+2}}{c-a}\Big)
    = \frac{ab(a^{p+1} - b^{p+1}) + ac (c^{p+1} - a^{p+1}) + bc (b^{p+1}-c^{p+1})}{(b-c)(b-a)(c-a)}
  \end{equation*}
we see that all three expressions are invariant under permutations of $a,b,c$,
i.e., any convenient ordering may be assumed. The left inequality
of the statement follows from an application of the mean value theorem, i.e.,
$$
\frac{1}{b-c}\Big(\frac{b^{p+2}-a^{p+2}}{b-a} -
\frac{c^{p+2}-a^{p+2}}{c-a}\Big) \leq (p+2)(p+1)\max\{a,b,c\}^p\leq (p+2)(p+1)(a^p+b^p+c^p).
$$
For the right inequality of the statement, notice that
\begin{align*}
\frac{1}{b-c}\Big(\frac{b^{p+2}-a^{p+2}}{b-a} -
\frac{c^{p+2}-a^{p+2}}{c-a}\Big) &= \frac{1}{b-c}\Big(\sum_{q=0}^{p+1}
(b^q-c^q)a^{p+1-q} \Big)\\
& = a^p+\underbrace{\sum_{q=2}^p \left[ \sum_{r=0}^{q-1}
b^r c^{q-1-r} \right]a^{p+1-q}}_{\geq 0}+\frac{b^{p+1}-c^{p+1}}{b-c}\\
& \geq a^p+b^p+c^p.
\end{align*}
\end{proof}

\begin{lemma}[Norm additivity] Let $\Omega\subset\R^2$ be an $2$-dimensional box such that $\overline{\Omega}=\bigcup_i \overline{\Omega}_i$ with at most countably many $2$-dimensional, mutually disjoint boxes $\Omega_i\subset\R^2$. Then the following holds:
\begin{alignat*}{5}
\|u\|_{L^p(\Omega_h)}^p &&& \leq \sum_i \|u\|_{L^p(\Omega_{i,h})}^p &&\leq 4\|u\|_{L^p(\Omega_h)}^p \\
\|u\|_{D(\Omega_h)}^2 &&& \leq \sum_i \|u\|_{D(\Omega_{i,h})}^2 && \leq 2\|u\|_{D(\Omega_h)}^2
\end{alignat*}
\label{additivity}
\end{lemma}
\begin{proof} The sum over all boxes $\Omega_{i,h}$ of the discrete $L^p$-norms includes all terms in the $L^p$-norm on $\Omega_h$, and each corner point is a member of at most $4$ boxes. This explains the inequality for the $L^p$-norms. Concerning the $D$-norm, again the sum over all boxes $\Omega_{i,h}$ of the discrete $D$-norms includes all terms in the $D$-norm on $\Omega_h$. This time, each discrete edge derivative occurs in at most $2$ boxes.
\end{proof}

\noindent
{\bf Remark.} Clearly, the above inequality has an $n$-dimensional version. The multiplicative factor in the $L^p$-norm becomes $2^n$ and in the $D$-norm it becomes $2^{n-1}$.

\begin{definition} For $n=2$ let $C_0=(0,h_1)\times(0,h_2)\subset\R^2$ be a unit
  mesh-box. Consider the following partition in
  triangles $\overline{C}_0 = T^1\cup T^2$ with
\begin{align*}
T^1 = \conv\left\{\begin{pmatrix} 0\\ 0 \end{pmatrix},
\begin{pmatrix} h_1\\ 0\end{pmatrix},
\begin{pmatrix} h_1 \\ h_2\end{pmatrix}\right\}, \qquad
T^2 =
\conv\left\{\begin{pmatrix} 0\\ 0 \end{pmatrix},
\begin{pmatrix} 0\\ h_2\end{pmatrix},
\begin{pmatrix} h_1 \\ h_2\end{pmatrix}\right\}.
\end{align*}
Note that a box $\overline{\Omega}$ as above can be decomposed in an at most
countable union of triangles which will be written as
$$
\overline{\Omega} = \bigcup_i  (T_i^1\cup T_i^2)
$$
where each $T_i^\alpha$ is a shift of $T^\alpha$, $\alpha=1,2$.
\end{definition}

\begin{lemma}[Interpolation lemma] Let $n=2$ and $\Omega, \Omega_h$ as
  above with $\overline{\Omega} = \bigcup_i (T_i^1\cup T_i^2)$. If $u:
  \overline{\Omega}_h\to [0,\infty)$ is a given function then there exists a
  continuous, piecewise on every $T_i^\alpha$ linear function $\tilde u:
  \overline{\Omega}\to [0,\infty)$ such that
\begin{equation}
u(y) = \tilde u(y) \mbox{ for all } y\in \overline{\Omega}_h
\label{values}
\end{equation}
and
\begin{equation}
D_i^+ u(x) =\frac{\partial \tilde u}{\partial x_i}(x)
\mbox{ for all } x \in \overline{\Omega}_{h} \setminus \partial^{+}_{i} \Omega_{h}
\mbox{ and all } 1 \le i \le n.
\label{derivatives}
\end{equation}
Moreover, for $p\in \N$
\begin{equation}
\frac{1}{8p^2} \sum_{x\in\overline{\Omega}_h} u(x)^p \h \leq \int_\Omega
\tilde u(x)^p\,dx \leq 8 \sum_{x\in\overline{\Omega}_h} u(x)^p \h.
\label{integrals_0}
\end{equation}
Finally
\begin{equation}
\int_{\Omega} |\nabla \tilde u|^2 \,dx \leq \sum_{i=1}^2 \sum_{x\in
  \overline{\Omega}_h\setminus \partial_i^+ \Omega_h}
|D_i^+ u(x)|^2 \h.
\label{integrals}
\end{equation}
\label{interpol}
\end{lemma}

\begin{proof}
Let $\tilde u$ be the linear interpolant between the values of $u$ at the corners of each
$T^\alpha_i$. Then (\ref{values}) and (\ref{derivatives}) follow
immediately. Due to the additivity property of the $L^p$ and $W^{1,p}$-norm from Lemma~\ref{additivity}, it is sufficient to consider first the statements \eqref{integrals_0}, \eqref{integrals} for the case when
$\Omega=C_0=(0,h_1)\times(0,h_2)$ is a mesh-box and
$\overline{C}_0=T^1\cup T^2$.

\medskip

So let us first consider (\ref{integrals_0}) where for
brevity we write $T^1=\conv\{A,B,D\}, T^2=\conv\{A,C,D\}$,
cf. Figure \ref{fg:RECTANGLE} with  $A=(0,0), B=(h_1,0), C=(0,h_2), D=(h_1,h_2)$.
\ifnum\pdfoutput=1
\begin{figure}[h]
\centering
\includegraphics{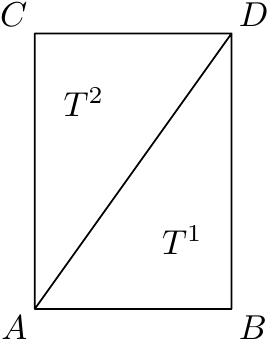}
\caption{}
\label{fg:RECTANGLE}
\end{figure}
\fi
On
$T^1$ the linear interpolant $\tilde u$ takes the form
$$
\tilde u(x_1,x_2) =
u(A)+x_1\frac{u(B)-u(A)}{h_1}+x_2\frac{u(D)-u(B)}{h_2}
$$
and
\begin{multline*}
\int_{T^1} \tilde u(x_1,x_2)^p d(x_1,x_2) = \\
\frac{h_1h_2}{(p+1)(p+2)}\frac{1}{u(D)-u(B)}
\Big(\frac{u(D)^{p+2}-u(A)^{p+2}}{u(D)-u(A)}-\frac{u(B)^{p+2}-u(A)^{p+2}}
{u(B)-u(A)}\Big).
\end{multline*}
Using the inequality of Lemma \ref{ineq} we find
$$
\big(u(A)^p+u(B)^p+u(D)^p\big)h_1h_2\geq \int_{T^1} \tilde
u(x_1,x_2)^p \,d(x_1,x_2) \geq
\frac{h_1h_2}{(p+1)(p+2)}\big(u(A)^p+u(B)^p+u(D)^p\big).
$$
A similar inequality holds for $T^2$. Since $\overline{C}_{0,h}=T^1\cup T^2$
we get
\begin{multline*}
\Big(2u(A)^p+u(B)^p+2u(D)^p+u(C)^p\Big)h_1h_2 \\
\geq \int_{C_0} \tilde
u(x_1,x_2)^p\,d(x_1,x_2) \geq \frac{h_1h_2}{(p+1)(p+2)}
\Big(2u(A)^p+u(B)^p+2u(D)^p+u(C)^p\Big).
\end{multline*}
Together with the inequality $(p+1)(p+2)\leq 8p^2$ we get an inequality similar to (\ref{integrals_0}) first for the mesh-box $C_0$ and then via Lemma~\ref{additivity} also for the general $2$-dimensional box by introducing another factor of $4$ for the upper bound.

\medskip

For the proof of (\ref{integrals}) note that
\begin{equation}
\int_{C_0} |\nabla \tilde u|^2\,dx = \frac{h_1h_2}{2}
\big(\bigl|\nabla\tilde u|_{T^1}\bigr|^2 + \bigl|\nabla \tilde u|_{T^1}\bigr|^2\big)
\label{int_1}
\end{equation}
and
\begin{equation}
\sum_{i=1}^2 \sum_{x\in
  \overline{C}_{0,h}\setminus \partial_i^+ C_{0,h}} |D_i^+ u(x)|^2\h =
h_1h_2\big(|D_1^+u(A)|^2+|D_1^+u(C)|^2+|D_2^+u(A)|^2+
|D_2^+u(B)|^2\big).
\label{int_2}
\end{equation}
Moreover, since $\tilde u$ is linear in $T^1, T^2$ and due to
(\ref{derivatives}) we have
\begin{align*}
\bigl|\nabla\tilde u|_{T^1}\bigr|^2 = |\partial_1 \tilde
u(A)|^2+|\partial_2\tilde u(B)|^2 = |D_1^+ u(A)|^2 +
|D_2^+u(B)|^2,\\
\bigl|\nabla\tilde u|_{T^2}\bigr|^2 = |\partial_1 \tilde
u(C)|^2+|\partial_2\tilde u(A)|^2 = |D_1^+ u(C)|^2 +
|D_2^+u(A)|^2.
\end{align*}
Adding the last two equations and using (\ref{int_1}), (\ref{int_2}) we find (\ref{integrals}) first for the mesh box $C_0$ (with a factor $1/2$ in the righ-hand side)
and then by summing over mesh-boxes and using Lemma~\ref{additivity}
also for the $2$-dimensional domain $\Omega$.
\end{proof}

The following lemma is well-known, cf. Gilbarg, Trudinger \cite{GT} or Adams \cite{adams}, but usually explicit estimates for the constants have to be extracted from the proof. We state the result with explicit bounds and give a proof.

\begin{lemma} Let $Q$ be a two-dimensional open rectangle with edge lengths $\alpha,\beta$ and $\tilde u\in W^{1,2}(Q)$. For every exponent $q\geq 1$ we have
$$
\int_Q |\tilde u|^{q}\,dx \leq \max\left\{2\sqrt{\pi}\frac{\alpha}{\beta}, 2\sqrt{\pi}\frac{\beta}{\alpha}, \frac{2}{\sqrt{|Q|}}\right\}^{q} |Q|\left(1+\frac{q}{2}\right)^{1+\frac{q}{2}}\|\tilde u\|_{W^{1,2}(Q)}^{q}.
$$
\label{classical_estimate}
\end{lemma}
\begin{proof} For every $x\in Q$ there exists an open rectangle $Q_x$ of half the edge length of
$Q$ such that $0$ is a corner of $Q_x$ and $x+Q_x\subset Q$. For $x\in Q$ and $y\in x+Q_x$ we have
$$
\tilde u(x) = \tilde u(y) + \int_0^1 \frac{d}{dt} \tilde u\bigl(tx+(1-t)y\bigr)\,dt = \tilde u(y) + \int_0^1 \nabla \tilde u\bigl(tx +(1-t)y\bigr) \cdot (x-y)\,dt.
$$
After taking absolute values and integrating with respect to $y\in x+Q_x$ we obtain
\begin{align*}
\frac{1}{4}|Q| |\tilde u(x)| & \leq \int_{x+Q_x} |\tilde u(y)|\,dy + \int_{x+Q_x} |x-y|\int_0^1 \left|\nabla \tilde u\bigl(tx +(1-t)y\bigr)\right|\,dt\,dy \\
& \leq \sqrt{|Q_x|} \|\tilde u\|_{L^2(Q)}+\int_0^1 \int_{Q_x} |z| |\nabla \tilde u(x+(1-t)z)|\,dz\,dt \\
& =\frac{\sqrt{|Q|}}{2}\|\tilde u\|_{L^2(Q)} + \int_0^1 \int_{(1-t)Q_x} |w| |\nabla \tilde u(x+w)| (1-t)^{-3}\,dw \,dt.
\end{align*}
The set $\{(t,w): 0\leq t \leq 1, w \in (1-t)Q_x\}$ is the same as
$$\left\{(t,w): w\in Q_x, 0\leq t \leq \min\left\{1-\frac{2|w_1|}{\alpha}, 1- \frac{2|w_2|}{\beta}\right\}\right\}.
$$
Hence, using Fubini's theorem we obtain
\begin{align*}
\frac{1}{4}|Q| |\tilde u(x)| &\leq \frac{\sqrt{|Q|}}{2}\|\tilde u\|_{L^2(Q)} + \int_{Q_x} \int_0^ {1-2\max\left\{\frac{|w_1|}{\alpha},\frac{|w_2|}{\beta}\right\}} |w| |\nabla \tilde u(x+w)| (1-t)^{-3}\,dt \,dw \\
& \leq \frac{\sqrt{|Q|}}{2}\|\tilde u\|_{L^2(Q)} + \frac{1}{8}\int_{Q_x} |w| |\nabla \tilde u(x+w)| \left(\max\left\{\frac{|w_1|}{\alpha},\frac{|w_2|}{\beta}\right\}\right)^{-2}\,dw \\
& \leq \frac{\sqrt{|Q|}}{2}\|\tilde u\|_{L^2(Q)} + \frac{1}{2}\max\{\alpha^2,\beta^2\} \int_{Q_x} |w| |\nabla \tilde u(x+w)| (|w_1|+|w_2|)^{-2}\,dw \\
& \leq \frac{\sqrt{|Q|}}{2}\|\tilde u\|_{L^2(Q)} + \frac{1}{2}\max\{\alpha^2,\beta^2\} \int_Q |w-x|^{-1} |\nabla \tilde u(w)|\,dw.
\end{align*}
Next we use the following two-dimensional estimate for Riesz-potentials,
cf. Gilbarg, Trudinger \cite{GT}, Section 7.8: for $x\in Q$, $f\in L^2(Q)$
define the Riesz-potential $(Rf)(x) := \int_Q |x-y|^{-1} |f(y)|\,dy$.
Then for all exponents $1\leq q<\infty$ we have
$$
\|Rf\|_{L^q(Q)} \leq \left(1+\frac{q}{2}\right)^{\frac{1}{2}+\frac{1}{q}}
                     \sqrt{\pi} |Q|^\frac{1}{q} \|f\|_{L^2(Q)}.
$$
Applying this to $f=|\nabla \tilde u|$ and using sublinearity of the norm we get
$$
\frac{1}{4}|Q| \|\tilde u\|_{L^q(Q)} \leq \frac{|Q|^{\frac{1}{2}+\frac{1}{q}}}{2}\|\tilde u\|_{L^2(Q)}+ \frac{1}{2}\max\{\alpha^2,\beta^2\} \left(1+\frac{q}{2}\right)^{\frac{1}{2}+\frac{1}{q}} \sqrt{\pi} |Q|^\frac{1}{q} \||\nabla \tilde u|\|_{L^2(Q)}
$$
and estimating
$\max\{\|\tilde{u}\|_{L^{2}}, \|\nabla \tilde{u}\|_{L^{2}}\} \le \|\tilde{u}\|_{W^{1,2}}$,
we obtain the claim:
$$
 \|\tilde u\|_{L^{q}(Q)}^{q} \leq
 \max\left\{2\sqrt{\pi}\frac{\alpha}{\beta}, 2\sqrt{\pi}\frac{\beta}{\alpha}, \frac{2}{\sqrt{|Q|}}\right\}^{q}
 |Q|\left(1+\frac{q}{2}\right)^{1+\frac{q}{2}}\|\tilde{u}\|_{W^{1,2}(Q)}^{q}.
$$
\end{proof}

\medskip

\noindent
\emph{Proof of Theorem~ \ref{sobolev_embedd_2d}:} We decompose $\R^2$ into mutually
  disjoint open rectangles $Q_i$ such that
  $\R^2=\bigcup_{i=1}^\infty \overline{Q}_i$ where up to a rigid motion each
  $Q_i$ is identical to a fixed box $Q := (0,2^kh_1)\times (0,2^lh_2)$. The integers $2^k,2^l$ are fixed such
  that $2^{-k} \leq h_1 < 2^{-k+1}$ and $2^{-l}\leq h_2 < 2^{-l+1}$ so that the edge lengths of $Q$ are between $1$ and $2$ and the area of $Q$ is between $1$ and $4$. In this way we also obtain the
  decomposition $\R^2_h = \bigcup_{i=1}^\infty\overline{Q}_{i,h}$. First we establish the inequality
  for $u\in W^{1,2}(\overline{Q}_h)$ with norms $\|\cdot\|_{A(\overline{Q}_h)}$ and
    $\|\cdot\|_{W^{1,2}(\overline{Q}_h)}$. The inequality \eqref{sobolev_2d} follows by using the basic
    inequality $A(st)\leq s^2 A(t)$ for $0\leq s<1$:
\begin{align*}
\sum_{x\in\R^n_h}
A\Big(\frac{u(x)}{C_S(2)\|u\|_{W^{1,2}(\R^2_h)}}\Big)\h &\leq
\sum_{i=1}^\infty \sum_{x\in \overline{Q}_{i,h}}
A\Big(\frac{u(x)}{C_S(2)\|u\|_{W^{1,2}(\R^2_h)}}\Big)\h\\
& \leq \sum_{i=1}^\infty
\frac{\|u\|^2_{W^{1,2}(\overline{Q}_{i,h})}}{4\|u\|^2_{W^{1,2}(\R^2_h)}}
\sum_{x\in \overline{Q}_{i,h}}
A\Big(\frac{2u(x)}{C_S(2)\|u\|_{W^{1,2}(\overline{Q}_{i,h})}}\Big)\h\\
& \leq \sum_{i=1}^\infty
\frac{\|u\|^2_{W^{1,2}(\overline{Q}_{i,h})}}{4\|u\|^2_{W^{1,2}(\R^2_h)}}\leq 1,
\end{align*}
provided $C_S(2)/2$ is the constant in (\ref{sobolev_2d}) for the rectangle
$\overline{Q}_h$. In other words: once we have the inequality \eqref{sobolev_2d}
for the rectangle $\overline{Q}_h$ then we obtain (\ref{sobolev_2d}) for $\R^2_h$ by doubling the constant.

\medskip

Now we prove (\ref{sobolev_2d}) for the rectangle $Q_h$. It suffices to give
  the proof for $u\geq 0$. The general case follows by replacing $u$ with $|u|$
  and using the reverse triangle inequality to verify that $|D_i^+ |u|(x)| \leq |D_i^+ u(x)|$.
  So let $u:\overline{Q}_h\to [0,\infty)$ be given and let $\tilde u$ be its linear
  interpolation from Lemma \ref{interpol}. Using Lemma~\ref{classical_estimate} for $Q$, $1\leq |Q|\leq 4$ and edge-lengths between $1$ and $2$ we have
$$
\int_Q \tilde u(x)^{2k} \,dx \leq 4(16 \pi)^k (1+k)^{1+k}\|\tilde
u\|_{W^{1,2}(Q)}^{2k} \leq  8(32\pi)^k k^{1+k}\|\tilde
u\|_{W^{1,2}(Q)}^{2k}.
$$
Using (\ref{integrals_0}) and (\ref{integrals})
of Lemma~\ref{interpol} we obtain
$$
\sum_{x\in \overline{Q}_h} u(x)^{2k}\h \leq 256(32\pi)^k k^{k+3}
\|u\|_{W^{1,2}(\overline{Q}_h)}^{2k}.
$$
Next we divide by $k!$ and use the inequalities (verified by induction)
$k^k/k! \leq e^{k-1}$, $k^3 \leq 4^k$ to obtain
$$
\sum_{x\in \overline{Q}_h} \frac{u(x)^{2k}}{k!}\h \leq \frac{256}{e}(128\pi e)^k
\|u\|_{W^{1,2}(Q_h)}^{2k}.
$$
Dividing again by $C'^{2k}\|u\|^{2k}_{W^{1,2}(\overline{Q}_h)}$
and summing from $k=1$ to $\infty$ we find
$$
\sum_{x\in \overline{Q}_h}
A\Big(\frac{u(x)}{C'\|u\|_{W^{1,2}(Q_h)}}\Big)\h \leq
\frac{256}{e}\sum_{k=1}^\infty \Big(\frac{128\pi e}{C'^2}\Big)^k.
$$
The choice of $C'^2=128\pi(e+256)$ makes the right hand side equal to
$1$. Thus, we get (\ref{sobolev_2d}) on $\overline{Q}_h$ with this
constant $C'$ and hence (\ref{sobolev_2d}) on $\R^2_h$ with
$C_S(2)=2C'=8\sqrt{2\pi(e+256)}$.

\medskip

Finally, for the proof of (\ref{power-sobolev_2d}) let $p\geq 2$ and
assume $u\in W^{1,2}(\R^2_h)$. Let
$\lfloor p\rfloor$ be the largest integer $\leq p$ and $\lceil
p\rceil$  be the smallest integer $\geq p$. Then $|u|^p\leq
|u|^{\lfloor p \rfloor}$ if $|u|\leq 1$ and $|u|^p\leq |u|^{\lceil p
  \rceil}$ if $|u|\geq 1$. Thus
$$
|u|^p = (u^2)^{p/2} \leq (u^2)^{\lfloor
  \frac{p}{2}\rfloor}+(u^2)^{\lceil \frac{p}{2}\rceil}\leq
(e^{u^2}-1)\Big(\left\lfloor \frac{p}{2}\right\rfloor!+
\left\lceil\frac{p}{2}\right\rceil!\Big)\leq
2(e^{u^2}-1)\left\lceil\frac{p}{2}\right\rceil!
$$
Using (\ref{sobolev_2d_explicit}) one finds
$$
\sum_{x\in\R^2_h} \left(\frac{|u(x)|}{C_S(2)\|u\|_{W^{1,2}}}\right)^p\h\leq
2\left\lceil\frac{p}{2}\right\rceil!\sum_{x\in\R^2_h}
A\left(\frac{|u(x)|}{C_S(2)\|u\|_{W^{1,2}}}\right)\h\leq
2\left\lceil\frac{p}{2}\right\rceil!.
$$
By using the inequality $(k!)^{1/k}\leq k$ and $\lceil
\frac{p}{2}\rceil/p\leq 1$ for $p\geq 2$ the previous inequality
implies
$$
\|u\|_{L^p} \leq
2^{1/p}C_S(2)\Big(\left\lceil\frac{p}{2}\right\rceil!\Big)^{1/p}\|u\|_{W^{1,2}}
\leq 2C_S(2)\left\lceil\frac{p}{2}\right\rceil\|u\|_{W^{1,2}} \leq
2C_S(2)p\|u\|_{W^{1,2}}.
$$
This finishes the proof of Theorem~\ref{sobolev_embedd_2d}.
\qed

\medskip

\begin{lemma}[Hardy's inequality in $1d$] \label{lm:HARDY_1D}
  Let $\Omega\subset \R$ be a bounded one-dimensional box. Then
    \begin{equation*}
      \sum_{x \in \Omega_{h}} \frac{u(x)^{2}}{\dist(x, \partial \Omega_{h})^{2}}
      \le 4 \sum_{x \in \overline{\Omega}_{h} \setminus \partial^{+} \Omega_{h}}
      \left( D^{+}_{h} u(x) \right)^{2}
    \end{equation*}
  for all $u \in W_{0}^{1,2}(\overline{\Omega}_{h})$.
\end{lemma}
\begin{proof} Without loss of generality we can assume $\Omega = (0, (l+1)h)$ with $l \in \N$.
  The proof consists of three steps.
  We begin with showing that it is sufficient to prove
    \begin{equation}\label{eq:HARDY_1D_FIRST}
      \sum_{k=1}^{s} \frac{u(kh)^{2}}{(kh)^{2}} \le
      4 \sum_{k=0}^{s-1} \left( \frac{u((k+1)h)-u(kh)}{h} \right)^{2}
    \end{equation}
  for all \map{u}{\{ 0,h,\dots, sh\}}{\R} with $u(0)=0$.
  Indeed, suppose this is proved. Then, for arbitrary $u\in W_0^{1,2}(\Omega_h)$, we obtain using (\ref{eq:HARDY_1D_FIRST})
    \begin{align*}
    \sum_{x \in \Omega_{h}} \frac{u(x)^{2}}{\dist(x, \partial \Omega_{h})^{2}}
       &\le \sum_{k=1}^{ \floor{l/2} } \frac{u(kh)^{2}}{(kh)^{2}}
      + \sum_{k=\ceil{l/2}}^{l} \frac{u(kh)^{2}}{((l+1)h - kh)^{2}}\\
      &= \sum_{k=1}^{ \floor{l/2} } \frac{u(kh)^{2}}{(kh)^{2}}
      + \sum_{j=1}^{l-\ceil{l/2} + 1} \frac{u((l+1-j)h)^{2}}{(jh)^{2}}\\
      &\le  4 \sum_{k=0}^{ \floor{l/2} - 1} (D^{+}_{h} u(kh))^{2}
      + 4 \sum_{j=0}^{ l- \ceil{l/2}} \left( \frac{u( (l-j+2)h ) - u((l-j+1)h)}{h} \right)^{2}\\
      &= 4 \sum_{k=0}^{ \floor{l/2} - 1} (D^{+}_{h} u(kh))^{2}
      + 4 \sum_{k=\ceil{l/2} }^{ l} \left( D^{+}_{h} u(kh) \right)^{2}
      \le 4 \sum_{k=0}^{l} (D^{+}_{h} u(kh))^{2}\\
      &= 4 \sum_{x \in \overline{\Omega}_{h} \setminus \partial^{+} \Omega_{h}}
      \left( D_{h}^{+} u(x) \right)^{2}.
      \end{align*}
Next we show that it does not restrict the generality to assume that $u$ in (\ref{eq:HARDY_1D_FIRST}) is non-decreasing, i.e. $u((k+1)h)-u(kh) \ge 0$, $0 \le k \le s-1$.
Indeed, let \map{u}{\{ 0,h, \dots, sh \}}{\R}, $u(0)=0$ be arbitrary
and define \map{v}{\{ 0, \dots, s h\}}{\R} by
  \begin{equation*}
  \begin{aligned}
    v(0) &:= u(0) = 0,\\
    v((k+1)h) &:= v(kh) + |u((k+1)h)-u(kh)|, \quad 0 \le k \le s-1.
  \end{aligned}
  \end{equation*}
By induction (assuming $|u(kh)|\leq v(kh)$) we see that
  \begin{equation*}
  \begin{aligned}
    |u(k+1)h| \le |u(kh)| + |u((k+1)h) - u(kh)| \le v((k+1)h)
  \end{aligned}
  \end{equation*}
i.e., $|u(kh)| \le v(kh)$ for all $k\in \{0,1,\ldots,s\}$. Since $v$ is non-decreasing this implies (using \eqref{eq:HARDY_1D_FIRST} for $v$) that
    \begin{equation*}
    \begin{aligned}
      \sum_{k=1}^{s} \frac{u(kh)^{2}}{(kh)^{2}}
      \le \sum_{k=1}^{s} \frac{v(kh)^{2}}{(kh)^{2}}
      \le 4\sum_{k=0}^{s-1} \left( D^{+}_{h} v(kh) \right)^{2}
      = 4\sum_{k=0}^{s-1} \left( D^{+}_{h} u(kh) \right)^{2}.
    \end{aligned}
    \end{equation*}
Finally, let a non-decreasing function $u:\{0,h,\ldots,sh\}\to \R$ with $u(0)=0$  be given.
For $a_k \geq 0$, $k=1,\ldots, s$, $ A_k := \sum_{i=1}^{k} a_{i}$ recall inequality No. 326 with $p=2$ from
\cite{HLP}:
\begin{equation}\label{eq:HARDY_HLP_INEQ}
  \begin{aligned}
    \sum_{k=1}^{s} \frac{A_{k}^{2}}{k^{2}} \le 4 \sum_{k=1}^{s} a_{k}^{2}.
  \end{aligned}
\end{equation}
If we set
$a_{k}:= u(kh) - u((k-1) h)\geq 0$ then we see that $A_{k} := \sum_{i=1}^{k} a_{i} = u(kh)$ and \eqref{eq:HARDY_HLP_INEQ} implies
  \begin{equation*}
    \sum_{k=1}^{s} \frac{u(kh)^{2}}{(kh)^{2}}
    \le 4 \sum_{k=1}^{s} \left( \frac{u(kh)-u((k-1)h)}{h} \right)^{2}
    = 4 \sum_{k=0}^{s-1} \left( D^{+}_{h} u(kh) \right)^{2},
  \end{equation*}
  which is the claimed inequality \eqref{eq:HARDY_1D_FIRST} for non-decreasing $u:\{0,h,\ldots,sh\}\to \R$ with $u(0)=0$
\end{proof}

\medskip

\noindent
\emph{Proof of Theorem~ \ref{th:HARDY_INEQUALITY}:}
We denote $d_{i}(x) := \dist(x, \partial_{i}^{-} \Omega_{h}
                           \cup \partial_{i}^{+} \Omega_{h} )$.
Using
  \begin{equation*}
    \sum_{i=1}^{n} \frac{1}{d^{2}_{i}(x)} \ge \max_{1 \le i \le n} \frac{1}{d_{i}(x)^{2}}
    = \frac{1}{\min_{1 \le i \le n} d_{i}(x)^{2}}
    = \frac{1}{\dist(x, \partial \Omega_{h})^{2}}
  \end{equation*}
and Lemma \ref{lm:HARDY_1D} we obtain
  \begin{align*}
  \sum_{x \in \Omega_{h}} \frac{u(x)^{2}}{\dist(x, \partial \Omega_{h})^{2}}
    & \le \sum_{i=1}^{n} \sum_{x \in \Omega_{h}} \frac{u(x)^{2}}{d_{i}(x)^{2}}
     = \sum_{i=1}^{n} \sum_{x \in \partial_{i}^{-}\Omega_h} \sum_{s=1}^{l_{i}-k_{i}-1}
      \frac{u(x+s \delta_{i})^{2}}{d_{i}(x + s \delta_{i})^{2}} \\
    &\le 4 \sum_{i=1}^{n} \sum_{x \in \partial_{i}^{-} \Omega_h} \sum_{s=0}^{l_{i}-k_{i}-1}
    \left( D^{+}_{i} u(x + s \delta_{i}) \right)^{2}
    = 4 \sum_{i=1}^{n} \sum_{x \in \overline{\Omega}_{h} \setminus \partial_{i}^{+} \Omega_{h}}
     \left( D_{i}^{+}(x) \right)^{2}.
  \end{align*}
\qed



\begin{thebibliography}{99}
\bibitem{adams} R.A. Adams, Sobolev spaces, Academic Press 1975.
\bibitem{brezis_marcus} H. Brezis and M. Marcus, {\sl Hardy's inequalities revisited.}
Ann. Scuola Norm. Sup. Pisa Cl. Sci. (4) {\bf 25} (1997), 217--237.
\bibitem{BT} H. Brezis and R.E.L. Turner, {\sl On a class of superlinear
elliptic problems}. Comm. Partial Differential Equations {\bf 2} (1977),

601--614.
\bibitem{GNN} B. Gidas, W.-M. Ni and L. Nirenberg, {\sl Symmetry and related properties
via the maximum principle}, Comm. Math. Phys. {\bf 68} (1979), 209--243.
\bibitem{GS} B. Gidas and J. Spruck, {\sl A priori bounds for positive
solutions of nonlinear elliptic equations}. Comm. Partial Differential
Equations {\bf 6} (1981), 883--901.
\bibitem{GT} G. Gilbarg and N.S. Trudinger, Elliptic partial
differential equations of second order, second edition, Springer 1983.
\bibitem{HLP} G. H. Hardy, J. E. Littlewood and G. Polya, Inequalities,
Cambridge University Press 1934. 
\bibitem{hmr} J.~Hor\'{a}k, P.J. McKenna and W. Reichel, {\sl Very weak solutions with boundary singularities for semilinear elliptic Dirichlet problems in domains with conical corners.} J. Math. Anal. Appl. {\bf 352} (2009), 496-514.
\bibitem{mat_sobo} T. Matskewich and P.E. Sobolevskii, {\sl The best possible constant in generalized Hardy's inequality for convex domain in $\R^n$.} Nonlinear Anal. {\bf 28} (1997), 1601--1610.
\bibitem{mcr} P.J. McKenna and W. Reichel, {\sl Gidas-Ni-Nirenberg results for
  finite difference equations: estimates of approximate symmetry}.  J. Math. Anal. Appl. {\bf 334}  (2007),  206--222.
\bibitem{mcr2} P.J. McKenna and W. Reichel, {\sl A priori bounds for semilinear equations and a new class of critical exponents for Lipschitz domains}.
  J. Funct. Anal. {\bf 244} (2007),  220--246.

\end{thebibliography}
\end{document}